\documentclass[12pt]{amsart}
\usepackage{enumerate}
\usepackage{hyperref}
\usepackage{mathabx}
\usepackage{blkarray}
\usepackage{tikz} 
\usetikzlibrary{backgrounds}
\usetikzlibrary{shapes}
\usetikzlibrary{decorations.markings}
\usepackage{subfig}
\usetikzlibrary{snakes}
\hypersetup{
    pdftitle=   {Unavoidable vertex-minors in large prime graphs},
   pdfauthor=  {O-joung Kwon and Sang-il Oum}
}
\newcommand\abs[1]{\lvert #1\rvert}

\usepackage{amsmath}
\newtheorem{THM}{Theorem}[section]
\newtheorem{LEM}[THM]{Lemma}
\newtheorem*{THMMAIN}{Theorem \ref{thm:mainthm}}
\newtheorem{COR}[THM]{Corollary}
\newtheorem{PROP}[THM]{Proposition}

\newtheorem{CLAIM}{Claim}

\theoremstyle{remark}

\theoremstyle{definition}

\newcommand\rank{\operatorname{rank}}
\newcommand\pivot{\wedge}
\newcommand\mat{\boxminus}
\newcommand\tri{\boxslash}
\newcommand\antimat{\boxtimes}
\def\K_#1{{K_{#1}}}
\def\S_#1{\overline{K_{#1}}}

\begin{document}
\title[Unavoidable vertex-minors]{Unavoidable vertex-minors\\ in large prime graphs}
\author{O-joung Kwon}
\author{Sang-il Oum}
\address{Department of Mathematical Sciences, KAIST, 291 Daehak-ro
  Yuseong-gu Daejeon, 305-701 South Korea}
\email{ojoung@kaist.ac.kr}
\email{sangil@kaist.edu}
\thanks{Supported by Basic Science Research
  Program through the National Research Foundation of Korea (NRF)
  funded by  the Ministry of Science, ICT \& Future Planning
  (2011-0011653).}
\date{\today}
\begin{abstract}
  A graph is \emph{prime} (with respect to the split decomposition)
  if its vertex set does not admit a partition $(A,B)$ (called a \emph{split}) 
  with $\abs{ A}$, $\abs{B} \ge 2$ such that 
  the set of edges joining $A$ and $B$ induces a complete bipartite
  graph.

  We prove that for each $n$,  there exists $N$ such that 
  every prime graph on at least $N$ vertices contains 
  a vertex-minor isomorphic to either a cycle of length $n$
  or  a graph consisting of two disjoint cliques of size $n$ joined by
  a matching.
\end{abstract}
\keywords{vertex-minor, split decomposition, blocking
  sequence, prime, generalized ladder}
\maketitle
\section{Introduction}\label{sec:intro}
In this paper, all graphs are simple and undirected.
We write 
$P_n$ and $C_n$ to denote a graph that is  a path and a cycle on $n$ vertices, respectively.
 We aim to find analogues of the following theorems. 
\begin{itemize}
\item (Ramsey's theorem)

  For every $n$, there exists $N$ such that
every graph on at least $N$ vertices contains 
an induced subgraph isomorphic to $\K_n$ or $\S_n$.
\item (folklore;  see Diestel's book \cite[Proposition 9.4.1]{Diestel2010}) 

  For every $n$, there exists $N$ such that 
  every \emph{connected} graph on at least $N$ vertices contains
  an induced subgraph isomorphic to 
  $K_n$, $K_{1,n}$, or $P_n$.
\item (folklore; see Diestel's book \cite[Proposition 9.4.2]{Diestel2010})

  For every $n$, there exists $N$ such that
  every \emph{$2$-connected} graph on at least $N$ vertices
  contains a topological minor isomorphic to 
  $C_n$ or $K_{2,n}$.
\item (Oporowski, Oxley, and Thomas \cite{OOT1993})

  For every $n$, there exists $N$ such that every \emph{$3$-connected}
  graph 
  on at least $N$ vertices
  contains a  minor isomorphic to  the wheel graph $W_n$ on $n$ vertices or $K_{3,n}$.
\item (Ding, Chen \cite{DC2004})

  For every integer $n$, there exists $N$ such that 
  every \emph{connected and co-connected} graph on at least $N$
  vertices
  contains an induced
  subgraph
  isomorphic to $P_n$, $K_{1,n}^s$ (the graph obtained from $K_{1,n}$
  by subdividing one edge once), 
  $K_{2,n}\setminus e$, 
  or 
  $K_{2,n}/e\setminus f \setminus g $ where $\{f,g\}$ is a matching in $K_{2,n}/e$.
  A graph is \emph{co-connected} if its complement graph is connected.

\item (Chun, Ding, Oporowski, and Vertigan~\cite{CDOV2009})

   For every integer $n\ge 5$, there exists $N$ such that
   every {internally $4$-connected} graph on at least $N$ vertices
   contains a {parallel minor} isomorphic to
   $K_n$, $K_{4,n}'$ ($K_{4,n}$ with a complete graph on the vertices of degree $n$),
   $TF_n$ (the $n$-partition triple fan with a complete graph on the vertices of degree $n$), 
   $D_n$ (the $n$-spoke double wheel), 
   $D_n'$ (the $n$-spoke double wheel with axle), 
   $M_n$ (the $(2n+1)$-rung Mobius zigzag ladder), 
   or $Z_n$ (the $(2n)$-rung zigzag ladder).

\end{itemize}
These theorems commonly state that every sufficiently large graph %
having 
certain connectivity 
contains at least one graph in the list of \emph{unavoidable}
graphs by certain graph containment relation.
Moreover in each theorem, the list of unavoidable graphs is \emph{optimal} in the sense that each unavoidable graph
in the list has the required connectivity, 
can be made arbitrary large, 
and does not contain other unavoidable graphs in the list.

In this paper, we discuss \emph{prime} graphs as a connectivity
requirement.
A \emph{split} of a graph $G$ is a partition $(A,B)$ of the vertex set $V(G)$
having subsets $A_0\subseteq A$, $B_0\subseteq B$ such that
$\abs{A},\abs{B}\ge 2$ and 
a vertex $a\in A$ is adjacent to a vertex $b\in B$
if and only if 
$a\in A_0$ and $b\in B_0$.
This concept was first studied by Cunningham~\cite{Cunningham1982} in
his research on split decompositions. 
We say that a graph is \emph{prime} if it has no splits. 
Sometimes we say a graph is \emph{prime with respect to split
decomposition} to distinguish with another notion of primeness with
respect to modular decomposition.

Prime graphs play important role in the study of circle graphs
(intersection graphs of chords in a circle) and
their recognition algorithms. Bouchet~\cite{Bouchet1987b},
Naji~\cite{Naji1985}, and Gabor, Hsu, and Supowit~\cite{GSH1989}
independently showed that 
prime circle graphs have a unique chord diagram.  This
is comparable to the fact that $3$-connected planar graphs have a
unique planar embedding.

The graph containment relation we will mainly discuss is
called a \emph{vertex-minor}. A graph $H$ is a \emph{vertex-minor} of
a graph $G$ if there exist a sequence $v_1,v_2,\ldots,v_n$ of (not
necessarily distinct) vertices and a subset $X\subseteq V(G)$ such
that
$H=G*v_1*v_2\cdots *v_n\setminus X$,
where $G*v$ is an operation called \emph{local complementation}, to
take the complement graph only in the neighborhood of $v$. The detailed
description will be given in Section~\ref{subsec:vmpm}.
Vertex-minors are important in circle graphs;
for instance, Bouchet~\cite{Bouchet1994} proved that a graph is a
circle graph if and only if it has no vertex-minor isomorphic to one
of three particular graphs.

Prime graphs have been studied with respect to vertex-minors,
perhaps because local complementation preserves prime graphs, shown by
Bouchet~\cite{Bouchet1987b}.
In addition, he showed the following. 
\begin{THM}[Bouchet~\cite{Bouchet1987b}]\label{thm:bouchet}
  Every prime graph on at least $5$ vertices must contain a
  vertex-minor isomorphic to $C_5$.
\end{THM}

Here is the main theorem of this paper.
\begin{THMMAIN}
  For every $n$, there is $N$ such that
  every prime graph on at least $N$ vertices has a vertex-minor
  isomorphic to $C_n$ or $\K_n\mat\K_n$. 
\end{THMMAIN}

\begin{figure}
 \tikzstyle{v}=[circle, draw, solid, fill=black, inner sep=0pt, minimum width=3pt]
  \centering
  \begin{tikzpicture}[scale=0.6]
    \node[v](v1) at (0,-1){};
    \node[v](v2) at (1,-1.5){};
    \node[v](v3) at (1.5,-3){};
    \node[v](v4) at (1,-4.5){};
    \node[v](v5) at (0,-5){};
    \node[v](w1) at (5-0,-1){};
    \node[v](w2) at (5-1,-1.5){};
    \node[v](w3) at (5-1.5,-3){};
    \node[v](w4) at (5-1,-4.5){};
    \node[v](w5) at (5-0,-5){};
      \foreach \i in {1,2,3,4} {
        \foreach \j in {\i,...,5} {
          \draw (v\i)--(v\j);
          \draw (w\i)--(w\j);
          }
          }
      \foreach \x in {1,...,5} {
      \draw[thick](v\x)--(w\x);
      }
    \end{tikzpicture}
  \caption{$\K_5\mat \K_5$.}
  \label{fig:k5}
\end{figure}

 The graph $\K_n\mat\K_n$ is a graph obtained by joining two copies of
$\K_n$ by a matching of size $n$, see Figure~\ref{fig:k5}. This notation will be explained in
Section~\ref{subsec:construction}.
In addition, we show that this list of unavoidable vertex-minors in Theorem~\ref{thm:mainthm} is optimal, which will
be discussed in Section~\ref{sec:nonequiv}.
We will heavily use Ramsey's theorem iteratively and so our bound $N$ is astronomical
in terms of $n$. 

The proof is splitted into two parts. 
\begin{enumerate}
\item We first prove that for each $n$, there exists $N$ such that
  every prime graph having an induced path of length $N$
  contains a vertex-minor isomorphic to $C_n$.
  (In fact, we prove that $N=\lceil 6.75n^7\rceil$.)
\item Secondly, we prove that for each $n$, there exists $N$ such that
  every prime graph on at least $N$ vertices
  contains a vertex-minor isomorphic to $P_n$ or $\K_n\mat \K_n$.
\end{enumerate}
To prove (1), we actually prove first that 
every sufficiently large generalized ladder, a  certain type of
outerplanar graphs, contains $C_n$ as a vertex-minor. This will be
shown in Section~\ref{sec:ladder}.
Then, we use the technique of blocking sequences developed by
Geelen~\cite{Geelen1995} to construct a large generalized ladder in a
prime graph having a sufficiently long induced path, shown in
Section~\ref{sec:path}.
Blocking sequences 
will be discussed and developed  in Section~\ref{sec:blocking}.
The second part (2) is discussed in Section~\ref{sec:main}, where we
iteratively use Ramsey's theorem to find a bigger configuration
called a broom inside
a graph.
In Section~\ref{sec:boundsize},  we give
similar theorems of this type on vertex-minors
with respect to less restrictive connectivity requirements.

\section{Preliminaries}\label{sec:prelim}

  	For $X\subseteq V(G)$, let $\delta_G(X)$ be the set of edges
  	having one end in $X$ and another end in $V(G)\setminus X$.
	Let $N_G(x)$ be the set of the neighbors of a vertex $x$ in $G$.
	For $X\subseteq V(G)$, let $G[X]$ be the induced subgraph of $G$ on the vertex set $X$.
	For two disjoint subsets $S,T$ of $V(G)$, let $G[S,T]=G[S\cup T]\setminus (E(G[S])\cup E(G[T]))$. 
	Clearly, $G[S,T]$ is a bipartite graph with the bipartition $(S,T)$.

\subsection{Vertex-minors}\label{subsec:vmpm}

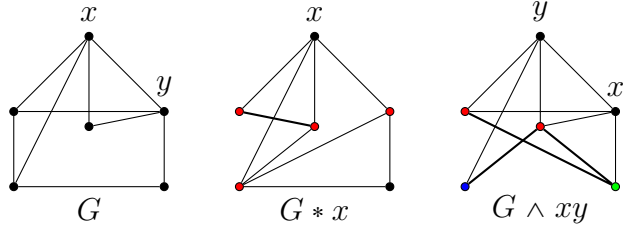
\begin{figure}
  \centering
  \tikzstyle{v}=[circle, draw, solid, fill=black, inner sep=0pt, minimum width=3pt]
  \begin{tikzpicture}
    \node [v,label=$x$] (x) at (1,2) {};
    \node [v,label=$y$] (y) at (2,1) {};
    \node [v] (d) at (2,0) {};
    \node [v] (a) at (0,1) {};
    \node [v] (b) at (0,0) {};
    \node [v] (c) at (1,0.8){};
    \draw (x)--(y)--(d)--(b);\draw (c)--(x)--(a)--(y);
    \draw(a)--(b)--(x);\draw(c)--(y);
    \node at (1,-.3) {$G$};

    \begin{scope}[xshift=3cm]
      \node [v,label=$x$] (x) at (1,2) {};
      \node [v,fill=red] (y) at (2,1) {};
      \node [v] (d) at (2,0) {};
      \node [v,fill=red] (a) at (0,1) {};
      \node [v,fill=red] (b) at (0,0) {};
      \node [v,fill=red] (c) at (1,0.8){};
      \draw (x)--(y)--(d)--(b);\draw (c)--(x)--(a);
      \draw (b)--(x);\draw [thick](c)--(a);\draw(y)--(b)--(c);
      \node at (1,-.3) {$G*x$};
    \end{scope}
    \begin{scope}[xshift=6cm]
      \node [v,label=$y$] (x) at (1,2) {};
      \node [v,label=$x$] (y) at (2,1) {};
      \node [v,fill=green] (d) at (2,0) {};
      \node [v,fill=red] (a) at (0,1) {};
      \node [v,fill=blue] (b) at (0,0) {};
      \node [v,fill=red] (c) at (1,0.8){};
      \draw (x)--(y)--(d);\draw (c)--(x)--(a)--(y);
      \draw(b)--(x);\draw(c)--(y);
      \draw[thick] (a)--(d)--(c);\draw[thick](b)--(c);
      \node at (1,-.3) {$G\pivot xy$};
    \end{scope}
  \end{tikzpicture}
  \caption{Local complementation and pivot.}
  \label{fig:lcpivot}
\end{figure}
The \emph{local complementation} of a graph $G$ at a vertex $v$
is an operation to replace the subgraph of $G$ induced by the
neighborhood of $v$ 
by its complement graph. 
In other words, to apply local complementation at $v$
for every pair $x$, $y$ of neighbors of $v$, 
we flip the pair $x$, $y$, 
where \emph{flipping} means that we delete the edge if it exists and
add it otherwise. 
We write $G*v$ to denote the graph obtained from $G$ by applying
local complementation of $G$ at $v$.
Two graphs are \emph{locally equivalent}
if one is obtained from another by applying a sequence of local
complementations. 
A graph $H$ is a \emph{vertex-minor} of $G$
if $H$ is an induced subgraph of a graph locally equivalent to $G$.

For an edge $xy$ of a graph $G$, 
a graph obtained by \emph{pivoting} an edge $xy$ of $G$
is defined as $G\pivot xy=G*x*y*x$.
Here is a direct way to see $G\pivot xy$; there are $3$ kinds of
neighbors of $x$ or $y$; some are adjacent to both, some are adjacent
to only $x$, others are adjacent to only $y$. We flip the adjacency
between all pairs of neighbors of $x$ or $y$ of distinct kinds
and then swap the two vertices $x$ and $y$.
Two graphs are \emph{pivot-equivalent} 
if one is obtained from another by a sequence of pivots.
Thus, pivot-equivalent graphs are locally equivalent.
See Figure~\ref{fig:lcpivot} for an example of these operations.

The following lemma by Bouchet provides a key tool to investigate
vertex-minors. His proof is based on isotropic systems, which are some
linear algebraic objects corresponding to the equivalence classes of
graphs with respect to local equivalence, introduced by Bouchet~\cite{Bouchet1987a}. A 
direct proof is given by Geelen and Oum~\cite{GO2009}.
\begin{LEM}[Bouchet~\cite{Bouchet1988}; see Geelen and Oum~\cite{GO2009}] \label{lem:bouchet1}
Let $H$ be a vertex-minor of $G$ and let $v\in V(G)\setminus V(H)$. Then $H$ is a vertex-minor of $G\setminus v$, $G\ast v\setminus v$, or $G\wedge vw\setminus v$ for a neighbor $w$ of $v$.
\end{LEM}
The choice of a neighbor $w$ in Lemma~\ref{lem:bouchet1} does not
matter, because if $x$ is adjacent to $y$ and $z$, then 
$G\pivot
xy=(G\pivot xz)\pivot yz$ (see \cite{Oum2004}).

\subsection{Cut-rank function}
Let $A(G)$ be the adjacency matrix of $G$ over the binary field.
For an $X\times Y$ matrix $A$, if $X'\subseteq X$ and $Y'\subseteq Y$,
then 
we write $A[X',Y']$ to denote the submatrix of $A$ 
obtained by taking rows in $X'$ and columns in $Y'$.

We define $\rho^*_G(X,Y)=\rank A(G) [X,Y]$.
This function satisfies the following submodular inequality (see Oum
and Seymour~\cite{OS2004}):
\begin{LEM}[{See Oum and Seymour~\cite{OS2004}}]\label{lem:submodular}
For all $A,B,A',B'\subseteq V(G)$,
\[\rho^*_G(A,B)+\rho^*_G(A',B')\ge \rho^*_G(A\cap A',B\cup B')+
\rho^*_G(A\cup A',B\cap B').\]
\end{LEM}

The \emph{cut-rank} function $\rho_G$ of a graph $G$ is 
defined as
\[\rho_G(X)=\rho^*_G(X,V(G)\setminus X)=\rank A(G)[X,V(G)\setminus X].\]
By Lemma \ref{lem:submodular}, we have the submodular inequality: 
\[
\rho_G(A)+\rho_G(B)\ge \rho_G(A\cap B)+\rho_G(A\cup B)
\]
for all $A,B\subseteq V(G)$.

The cut-rank function is
invariant under taking local complementation, which makes it useful
for us.
\begin{LEM}[Bouchet~\cite{Bouchet1989a}; See Oum~\cite{Oum2004}]\label{lem:cutrank}
  If $G$ and $H$ are locally equivalent, then 
  $\rho_G(X)=\rho_H(X)$
  for all $X\subseteq V(G)$.
\end{LEM}

\begin{LEM}[{Oum~\cite[Lemma 4.4]{Oum2004}}]\label{lem:reduce-submodular}
  Let $G$ be a graph and $v\in V(G)$. 
  Suppose that $(X_1,X_2)$, $(Y_1,Y_2)$ are partitions of
  $V(G)\setminus\{v\}$. Then 
  we have
  \[
  \rho_{G\setminus v}(X_1)+\rho_{G*v\setminus v}(Y_1)
  \ge \rho_G(X_1\cap Y_1)+\rho_G(X_2\cap Y_2)-1.
  \]
  Similarly if $w$ is a neighbor of $v$, then 
  \[
  \rho_{G\setminus v}(X_1)+\rho_{G\pivot vw\setminus v}(Y_1)
  \ge \rho_G(X_1\cap Y_1)+\rho_G(X_2\cap Y_2)-1.
  \]
\end{LEM}
Lemma~\ref{lem:reduce-submodular} is equivalent to the following lemma, which we will
use in the proof of Proposition~\ref{prop:reduce-blocking}.
\begin{LEM}\label{lem:reduce-submodular-2}
  Let $G$ be a graph and $v\in V(G)$. 
  Suppose that $X_1$, $X_2$, $Y_1$,  $Y_2$ are subsets of $V(G)\setminus\{v\}$
  such that $X_1\cup X_2=Y_1\cup Y_2$ and 
  $X_1\cap X_2=Y_1\cap Y_2=\emptyset$.
  Then 
  \begin{multline*}
  \rho^*_G(X_1,X_2)+\rho^*_{G*v}(Y_1,Y_2)\\\ge
  \rho^*_G(X_1\cap Y_1,X_2\cup Y_2\cup\{v\})
  +\rho^*_G(X_1\cup Y_1\cup \{v\},X_2\cap Y_2)-1.
  \end{multline*}
  Similarly if $w\in X_1\cup X_2$ is a neighbor of $v$, then 
  \begin{multline*}
  \rho^*_G(X_1,X_2)+\rho^*_{G\pivot vw}(Y_1,Y_2)\\\ge
  \rho^*_G(X_1\cap Y_1,X_2\cup Y_2\cup\{v\})
  +\rho^*_G(X_1\cup Y_1\cup \{v\},X_2\cap Y_2)-1.
  \end{multline*}
\end{LEM}
\begin{proof}
  Apply 
  Lemma~\ref{lem:reduce-submodular} with $G'=G[X_1\cup X_2\cup\{v\}]$.
\end{proof}

\subsection{Prime graphs}
For a graph $G$, a partition $(A,B)$ of $V(G)$ is called a
\emph{split}
if $\abs{A},\abs{B}\ge 2$
and there exist $A'\subseteq A$ and $B'\subseteq B$
such that $x\in A$ is adjacent to $y\in B$
if and only if $x\in A'$ and $y\in B'$.
A graph is \emph{prime} (with respect to the split decomposition)
if it has no splits. 
These concepts were introduced by Cunningham~\cite{Cunningham1982}.

Alternatively, a split can be understood with the \emph{cut-rank}
function $\rho_G$.
A partition $(A,B)$ of $V(G)$ is a split
if and only if 
$\abs{A},\abs{B}\ge 2$ 
and yet $\rho_G(A)\le 1$.

The following lemma is natural.
\begin{LEM}\label{lem:primeinduced}
  If a prime graph $H$ on at least $5$ vertices
  is a vertex-minor of a graph $G$, then 
  $G$ has a prime induced subgraph $G_0$ 
  such that $G_0$ has a vertex-minor isomorphic to $H$.
\end{LEM}
\begin{proof}
  We may assume that $G$ is connected.
  It is enough to prove the following claim:
  if $G$ has a split $(A,B)$, then there exists a vertex $v$ such
  that $H$ is isomorphic to a vertex-minor of $G\setminus v$.
  Let $G'$ be a graph locally equivalent to $G$ such that $H$ is an
  induced subgraph of $G'$.
  We have 
  $\rho_H(V(H)\cap A)
  = \rho_{G'}^*(V(H)\cap A,V(H)\cap B)
  \le \rho_{G'}^*(A,B)\le 1$ and therefore 
  $\abs{V(H)\cap A}\le 1$ or $\abs{V(H)\cap B}\le 1$ because $H$ is
  prime. 
  By symmetry, let us assume $\abs{V(H)\cap B}\le 1$.
  Let us choose $x\in B$ such that 
  $x$ has a neighbor in $A$ and 
  $x\in V(H)$ if $V(H)\cap B$ is nonempty.

  Let $H'$ be a vertex-minor of $G$ on $A\cup\{x\}$ such that
  $H$ is isomorphic to a vertex-minor of $H'$.
  Then $H'=G*v_1*v_2\cdots *v_n\setminus (B\setminus\{x\})$
  for some sequence $v_1,v_2,\ldots,v_n$ of vertices.
  We may choose $H'$ and $n$ so that $n$ is minimized.
  
  Suppose $n>0$.
  Then $v_n\in B\setminus \{x\}$.
  Let $H_0=G*v_1*v_2\cdots  *v_{n-1}\setminus (B\setminus\{x,v_n\})$.
  Since $(A,\{x,v_n\})$ is a split of $H_0$,
  one of the following holds.
  \begin{enumerate}[(i)]
  \item   The two vertices $v_n$ and $x$ have the same set of neighbors in $A$.
  \item The vertex $v_n$ has no neighbors in $A$.
  \item The vertex $x$ has no neighbors in $A$.
  \end{enumerate}
  If we have the case (i), then 
  $(H_0\setminus v_n) *x =H'$ and therefore $H$ is
  isomorphic to a
  vertex-minor of $H_0\setminus  v_n$, contradicting our
  assumption that $H$ is chosen to minimize $n$.
  If we have the case (ii),  then 
  $H_0\setminus v_n=H'$, contradicting the assumption too.
  Finally if we have the case (iii), then 
  $x$ is adjacent to $v_n$ in $G$ because $G$ is connected.
  Then $H_0*v_n\setminus v_n$ is isomorphic to $H_0*v_n\setminus x$.
  Then $H_0\setminus x$ has a vertex-minor isomorphic to $H$,
  contradicting our assumption that $n$ is minimized.
\end{proof}

\subsection{Constructions of graphs}
\label{subsec:construction}
For two graphs $G$ and $H$ on the same set of $n$ vertices, we would like
to introduce operations to construct  graphs on $2n$ vertices by making the disjoint
union of them and adding some edges between two graphs. 
Roughly speaking, $G\mat H$ will add a perfect matching,
 $G\antimat H$ will add the complement of a perfect matching, 
 and $G\tri H$ will add a bipartite chain graph.
Formally, 
for two graphs $G$ and $H$ on $\{v_1,v_2,\ldots,v_n\}$, 
let 
$G\mat H$, $G\antimat H$, $G\tri H$
be graphs on $\{v_1^1,v_2^1,\ldots,v_n^1,
v_1^2,v_2^2,\ldots, v_n^2\}$
such that  for all $i,j\in \{1,2,\ldots,n\}$, 
\begin{enumerate}[(i)]
\item 
$v_i^1v_j^1\in E(G\mat H)$ if and only if $v_i v_j\in E(G)$,
\item 
$v_i^2v_j^2\in E(G\mat H)$ if and only if $v_i v_j\in E(H)$,
\item 
$v_i^1v_j^2\in E(G\mat H)$ if and only if $i=j$,
\item 
$v_i^1v_j^1\in E(G\antimat H)$ if and only if $v_i v_j\in E(G)$,
\item 
$v_i^2v_j^2\in E(G\antimat H)$ if and only if $v_i v_j\in E(H)$,
\item 
$v_i^1v_j^2\in E(G\antimat H)$ if and only if $i\neq j$,
\item 
$v_i^1v_j^1\in E(G\tri H)$ if and only if $v_i v_j\in E(G)$,
\item 
$v_i^2v_j^2\in E(G\tri H)$ if and only if $v_i v_j\in E(H)$,
\item 
$v_i^1v_j^2\in E(G\tri H)$ if and only if $i\ge j$.
\end{enumerate}
See Figure~\ref{fig:construction} for $\K_5\mat\S_5$,
$\K_5\antimat\S_5$, and $\K_5\tri \S_5$.
\begin{figure}
 \tikzstyle{v}=[circle, draw, solid, fill=black, inner sep=0pt, minimum width=3pt]
  \centering
  \newcommand\Sfive[1]{    \begin{tikzpicture}[scale=0.6]
      \foreach \x in {1,...,5} {
        \node [v]  (v\x) at(0,-\x){};
        \node [v]  (w\x) at (2,-\x){};
        \draw (-.5,-\x) node [left] {$v^1_\x$};
        \draw (w\x) node [right] {$v^2_\x$};
      }
      \draw (v1)--(v5);
      \draw(v1) [in=120,out=-120] to (v3);
      \draw(v2) [in=120,out=-120] to (v4);
      \draw(v3) [in=120,out=-120] to (v5);
      \draw(v1) [in=120,out=-120] to (v4);
      \draw(v2) [in=120,out=-120] to (v5);
      \draw(v1) [in=120,out=-120] to (v5);
      \foreach \x in {1,...,5} 
      \foreach \y in {1,...,5} {
        #1
      }
    \end{tikzpicture}
    }
    \Sfive{       \ifnum \x=\y  \draw (v\x)--(w\y);       \fi      }
    $\qquad$
    \Sfive{       \ifnum \x=\y  \else      \draw (v\x)--(w\y);       \fi      }
    $\qquad$
    \Sfive{       \ifnum \x>\y   \draw (v\x)--(w\y);       
    \else \ifnum\x=\y \draw (v\x)--(w\y);\fi\fi}
  \caption{$\K_5\mat \S_5$, $\K_5\antimat \S_5$, and $\K_5\tri \S_5$.}
  \label{fig:construction}
\end{figure}

We will use the following lemmas.

\begin{LEM}\label{lem:lengthonecase1} Let $n\ge 3$ be an integer.
\begin{enumerate}
\item  $\K_n\antimat\S_n$ has a vertex-minor isomorphic to $\K_{n-1}\mat\K_{n-1}$.
\item $\S_n\antimat\S_n$ has a vertex-minor isomorphic to $\K_{n-2}\mat\K_{n-2}$.
\end{enumerate}
\end{LEM}
\begin{proof}
(1) Let $V(\K_n)=V(\S_n)=\{v_i:1\le i\le n\}$.
	The graph $(\K_n\antimat\S_n)*v^1_1*v^2_1\setminus v^1_1 \setminus v^2_1$ is isomorphic to $\K_{n-1}\mat\K_{n-1}$.
     
    (2) Let $V(\S_n)=\{v_1,v_2, \ldots, v_n\}$. 
    The graph $(\S_n\antimat\S_n)*v^1_1\setminus v^1_1 \setminus v^2_1$ is isomorphic to $\S_{n-1}\mat\K_{n-1}$.
    By (1), $\S_n\antimat\S_n$ has a vertex-minor isomorphic to $\K_{n-2}\mat\K_{n-2}$.
   \end{proof}

\begin{LEM}\label{lem:lengthonecase2} Let $n$ be a positive integer.
\begin{enumerate}
	\item The graph $\S_{n}\tri\S_{n}$ is pivot-equivalent to $P_{2n}$.
	\item The graph $\K_{n}\tri\S_{n}$ is locally equivalent to $P_{2n}$.
\end{enumerate}
\end{LEM}

\begin{proof}
	(1) Let $P=p_1p_2 \ldots p_{2n}$. 
	We can check that $\S_{n}\tri\S_{n}$ can be obtained from $P$ by pivoting $p_ip_{i+1}$ for all $i=1,3,\ldots, 2n-1$.	

	(2) Let $V(\K_{n})=V(\S_{n})=\{v_1, v_2, \ldots, v_n\}$. 
	Since $(\K_{n}\tri\S_{n})*v^2_1$ is isomorphic to $\S_{n}\tri\S_{n}$, 
	the result follows from (1).
\end{proof}

\subsection{Ramsey numbers}
	A \emph{clique} is a set of pairwise adjacent vertices.
        A \emph{stable set} or an
\emph{independent set} is a set of pairwise non-adjacent vertices.

We write  $R(n_1,n_2, \ldots, n_k)$ to denote
the minimum number $N$ such that
in every $k$ coloring of the edges of $K_N$, 
there exist $i$ and a clique of size $n_i$ 
whose edges are all colored with the $i$-th color.
Such a number exists by Ramsey's theorem~\cite{Ramsey1930}.

\section{Unavoidable vertex-minors in large
  graphs}\label{sec:boundsize}
We present three simple statements on unavoidable vertex-minors. These
are optimal as discussed in Section~\ref{sec:intro}.
\begin{THM}\label{thm:boundedsize}
  \begin{enumerate}
  \item For every $n$, there exists $N$ such that 
  every graph on at least $N$ vertices
  has  a vertex-minor isomorphic to $\S_n$.

  \item  For every $n$, there exists $N$ such that 
  every connected graph having at least $N$ vertices
  has a vertex-minor isomorphic to $\K_n$.

  \item For every $n$, there exists $N$ such that
  every graph having at least $N$ edges
  has a vertex-minor isomorphic to $K_n$ or $\S_n\mat\S_n$.
  \end{enumerate}
\end{THM}
\begin{proof}
(1)  If a graph has no $\S_n$ as a vertex-minor, then it has no
  vertex-minor isomorphic to $\K_{n+1}$.
  So we can take $N=R(n,n+1)$.

(2)  Let us assume that $G$ has no vertex-minor
  isomorphic to $K_n$.
  Then the maximum degree of $G$ is less than $\Delta=R(n-1,n-1)$ 
  by Ramsey theorem.
  If $\abs{V(G)}$ is big enough, 
  then it contains an induced path $P$ of length $2n-3$
  because the maximum degree is bounded.
  By Lemma~\ref{lem:lengthonecase2}, 
  $P_{2n-2}$
  has a vertex-minor isomorphic to 
  $\K_{1,n-1}$, that is locally equivalent to $\K_{n}$.
(3)  Let $G$ be a graph having no vertex-minor isomorphic to $\K_n$ or $\S_n\mat\S_n$.
  Each component of $G$ has bounded number of vertices, say $M$, by (2).
  Since $\S_n\mat\S_n$ is not a vertex-minor of $G$, 
  $G$ has less than $n$ 
  non-trivial components. (A component is trivial if it has no edges.)
  So $G$ has at most $\binom{M}{2} (n-1)$ edges.
\end{proof}

\section{Obtaining a long cycle in a huge generalized ladder}\label{sec:ladder}
\begin{figure}
  \centering
  \tikzstyle{v}=[circle, draw, solid, fill=black, inner sep=0pt, minimum width=3pt]
  \begin{tikzpicture}
    \foreach \x in {1,...,9}
    {
      \node [v,label=below:$p_{\x}$] (p\x) at (\x,0) {};
    }
    \foreach \x in {1,...,8}
    {
      \node [v,label=$q_{\x}$] (q\x) at (\x,1) {};
    }
    \draw [thick](p1)--(p9);
    \draw [thick] (q1)--(q8);
    \foreach \x in {1,3} {\draw (p1)--(q\x);}
    \foreach \x in {3,4,5} {\draw (p3)--(q\x);}
    \foreach \x in {5,6} {\draw (p4)--(q\x);}
    \foreach \x in {6,7,8} {\draw (q6)--(p\x);}
    \foreach \x in {8,9}{\draw(q8)--(p\x);}
  \end{tikzpicture}
  \caption{An example of a generalized ladder.}
  \label{fig:ladder}
\end{figure}
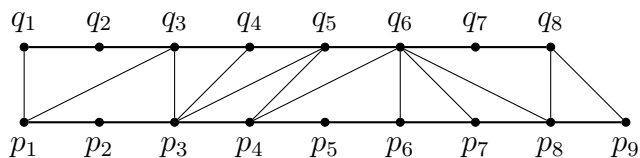
A \emph{generalized ladder} is a graph $G$ with two vertex-disjoint paths
$P=p_1p_2\ldots p_a$, $Q=q_1q_2\ldots q_b$ ($a,b\ge 1$)
with additional edges, called \emph{chords}, each joining a vertex of $P$ with a vertex of $Q$
such that $V(P)\cup V(Q)=V(G)$,
$p_1$ is adjacent to $q_1$,
$p_a$ is adjacent to $q_b$,
and no two chords cross. Two chords $p_iq_j$ and $p_{i'}q_{j'}$ ($i<i'$) \emph{cross} if
and only if $j>j'$.
We remark that 
a generalized ladder is a outerplanar graph whose weak dual is a path.
We call $p_1q_1$ the \emph{first chord}
and $p_aq_b$ the \emph{last chord} of $G$.
Since no two chords cross, $p_1$ or $q_1$ has degree at most $2$.
Similarly, $p_a$ or $q_b$ has degree at most $2$. 
See Figure~\ref{fig:ladder} for an example.

We will prove the following proposition.
\begin{PROP}\label{prop:ladder}
  Let $n\ge 2$.
  Every generalized ladder having at least 
  $4608n^5$
  vertices 
  has a cycle of length $4n+3$ as  a vertex-minor.
\end{PROP}
\subsection{Lemmas on a fan}\label{subsec:fan}
Let $F_n$ be a graph on $n$ vertices with a specified vertex $c$, called the center,
such that $F_n\setminus c$ is a path on $n-1$ vertices
and $c$ is adjacent to all other vertices.
We call $F_n$ a \emph{fan} on $n$ vertices.
\begin{LEM}\label{lem:fan}
 A fan $F_{3n}$ has a vertex-minor isomorphic to a
 cycle of length $2n+1$.
\end{LEM}
\begin{proof}
  Let $c$ be the center of $F_{3n}$.
  Let $v_1,v_2,\ldots,v_{3n-1}$ be the non-center vertices in
  $F_{3n}$ forming a path.
  Let $G=F_{3n}*v_3*v_6*v_9\cdots*v_{3n-3}$.
  Clearly $c$ is adjacent to $v_i$ in $G$ if and only if 
  $i\in \{1,3n-1\}$ or 
  $i\equiv 0\pmod 3$ and furthermore $v_{3i-1}$ is adjacent to
  $v_{3i+1}$ in $G$ for all $i$.
  Let $H=G\setminus \{v_3,v_6,\ldots,v_{3n-3}\}$.
  Then  $H$ is a cycle of length $3n-(n-1)$.
\end{proof}

\begin{LEM}\label{lem:pathone}
  Let $n\ge 2$.
  Let $G$ be a graph with  a vertex $c$ 
  such that $G\setminus c$ is isomorphic to an induced path $P$
  whose both ends are adjacent to $c$.
  If $\abs{V(G)}\ge 6(n-1)^2-3$,
  then  $G$ has a vertex-minor isomorphic to a cycle of length $2n+1$.
\end{LEM}
\begin{proof}
  We may assume that $n\ge 3$.
  Let $P=v_1v_2\ldots v_k$ with $k\ge 6$.
  We may assume that $v_2$ is adjacent to $c$ because otherwise we
  replace $G$ with $G*v_1$.
  Similarly we may assume that $v_{k-1}$ is adjacent to $c$. 
  We may also assume $v_3$ is adjacent to $c$ because otherwise we
  replace $G$ with $G\pivot v_1v_2$.
  Similarly we may assume that $v_{k-2}$ is adjacent to $c$.

  If $c$ is adjacent to at least $3n-1$ vertices on
  $P$, then $G$ has a vertex-minor isomorphic to $F_{3n}$.
  So by Lemma~\ref{lem:fan}, $G$ has a vertex-minor isomorphic
  to a cycle of length $2n+1$.
  Thus we may assume that the number of neighbors of $c$ is at most $3n-2$.
  The neighbors of $c$ gives a partition of $P$ into 
  at most $3n-3$ subpaths.
  We already have $4$ subpaths at both ends having length $1$.
  Since \[
  \abs{E(P)}\ge 
  6(n-1)^2-3-2>  (2n-2)((3n-3)-4)+4, \]
  there exists a subpath $P'$ of $P$ having length at least $2n-1$
  such that no internal vertex of $P'$ is adjacent to $c$
  and the ends of $P'$ are adjacent to $c$.
  This together with $c$ gives an induced cycle of length at least $2n+1$.
\end{proof}

\subsection{Generalized ladders of maximum degree at most $3$}

\begin{LEM}\label{lem:ladder-deg3}
  Let $G$ be a generalized ladder of maximum degree $3$.
  If $G$ has at least $6n$ vertices of degree $3$, 
  then $G$ has a cycle of length $4n+3$ as a vertex-minor.
\end{LEM}
\begin{proof}
  We proceed by induction on $\abs{V(G)}$.
  Let $P$, $Q$ be two defining paths of $G$.
  We may assume that all internal vertices of $P$ or $Q$ has degree
  $3$, because if $P$ or $Q$ has an internal vertex $v$ of degree $2$,
  then we apply the induction hypothesis to $G*v\setminus v$.
  Since $p_1$ or $q_1$ has degree $2$, 
  we may assume that $p_1$ has degree $2$ by symmetry.
  We may assume that $q_1$  has degree $3$ because otherwise we can
  apply the induction hypothesis to $G*q_1\setminus q_1$.
  Consequently $q_1$ is adjacent to $p_2$ and thus for each internal
  vertex $q_i$ of $Q$, $q_i$ is adjacent to $p_{i+1}$
  and each internal vertex $p_{i+1}$ of $P$ is adjacent to $q_i$.
  Thus either $a=b$ and $p_a$ has degree $3$
  or $a=b+1$ and $p_a$ has degree $2$.
  But if $a=b+1$ and $p_a$ has degree $2$, then we can apply the induction
  hypothesis to $G*p_a\setminus p_a$. 
  Thus we may assume that $a=b$ and $p_a$ has degree $3$.
  Since $G$ has at least $6n$ vertices of degree $3$, $a> 3n$ and
  $b>3n$.
  If $a=b>3n+1$, then we can apply the induction hypothesis to
  $G\setminus q_b$.
  Thus we may assume that $a=b=3n+1$
  and $p_a$ has degree $3$ and $q_b$ has degree $2$.
  Note that $p_i $ is adjacent to $q_{i-1}$ for all $i=2,\ldots,3n+1$.
  Then $G*p_1\pivot p_4q_3 \pivot p_7q_6 \cdots \pivot p_{3n+1}q_{3n}
  \setminus
  \{p_4,p_7,\ldots,p_{3n-2},q_3,q_6,\ldots,q_{3n-3},q_{3n+1}\}$ 
  is isomorphic to a cycle of length $4n+3$.
\end{proof}
\begin{LEM}\label{lem:ladder-final}
  Let $G$ be a generalized ladder of maximum degree $3$. 
  If $\abs{V(G)}\ge12n^2$, then 
  $G$ has a cycle of length $4n+3$ as a vertex-minor.
\end{LEM}
\begin{proof}
  Let $P$, $Q$ be two defining paths of $G$. 
  We may assume $a>1$ and $b>1$ because
  otherwise $G$ has an induced cycle of length at least $6n^2+1\ge
  4n+3$.

  Let $p_xq_y$ be the unique chord other than $p_1q_1$ with minimum $x+y$.
  We claim that we may assume $(x-1)+(y-1)\le 2$.
  Suppose not. Then $p_xq_y$, $p_1q_1$ and subpaths of $P$ and $Q$
  form a cycle of length $x+y\ge 5$ and
  $p_1,p_2,\ldots,p_{x-1},q_1,q_2,\ldots,q_{y-1}$ have degree $2$.
  By moving the first few vertices of $P$ to $Q$ or $Q$ to $P$,
  we may assume that $x\ge 3$ and $y\ge 2$.
  Then we may replace $G$ with $G*p_1$. This proves the claim.

  Thus the induced cycle containing $p_1q_1$ has at most $2$ edges
  from $E(P)\cup E(Q)$.
  Similarly we may assume that 
  the induced cycle containing $p_aq_b$ has at most $2$ edges
  from $E(P)\cup E(Q)$.

  If $G$ has at least $6n$ vertices of degree $3$,
  then by Lemma~\ref{lem:ladder-deg3}, we obtain a desired
  vertex-minor.
  So we may assume that $G$ has at most $6n-1$ vertices of degree
  $3$. 
  Thus $G$ has at most $3n-1$ chords other  than $p_1q_1$ and
  $p_aq_b$.
  These chords give at most $3n$ induced cycles of $G$
  where each edge in $E(P)\cup E(Q)$ appears in
  exactly one of them.
  If every such induced cycle has length at most $4n+2$, then 
  \[ \abs{E(P)\cup E(Q)} \le (3n-2)(4n)+4 = 12n^2-8n+4 <12n^2-2.\]
  Since $\abs{V(G)}\ge 12n^2$, 
  we have $\abs{E(P)\cup E(Q)}\ge 12n^2-2$. This leads to a contradiction.
\end{proof}
\subsection{Generalized ladders of maximum degree $4$}
\begin{LEM}\label{lem:removedeg4}
  Let $G$ be a generalized ladder
  of maximum degree at most $4$. 
  Let $\alpha$ be the number of vertices of $G$ having degree $3$ or $4$.
  Then $G$ has a vertex-minor $H$ that is a generalized ladder
  of maximum degree at most $3$
  such that $\abs{V(H)}\ge  \alpha/4$.
\end{LEM}
\begin{proof}
  Let $P=p_1p_2\ldots p_a$, $Q=q_1q_2\ldots q_b$ be the paths defining
  a generalized ladder $G$.
  Let $X_{i,j}=\{p_1,p_2,\ldots,p_i,q_1,q_2,\ldots,q_j\}$.
  We may assume $\alpha>8$.

  If $a=1$, then $p_1$ has at least $\alpha-1$ neighbors but the
  maximum degree is $4$ and therefore $\alpha\le 5$, contradicting our assumption.
  Thus $a>1$. Similarly $b>1$.

  We may also assume that no internal vertex of $P$ or $Q$ has degree
  $2$, because otherwise we can apply local complementation and remove
  it.

  Let $\alpha_{i,j}(G)$ be the number of vertices in $V(G)\setminus
  X_{i,j}$
  having degree $3$ or $4$. 
  We will prove the following.
  \begin{CLAIM}\label{claim1}
    Suppose that there exist $1\le i<a$ and $1\le j<b$ such that 
    $\delta_G(X_{i,j})$ has exactly two edges
    and 
    every vertex in $X_{i,j}$ has degree $2$ or $3$ in $G$.
    Then $G$ has a vertex-minor $H$ that is a generalized ladder of
    maximum degree at most $3$ such that
    $\abs{V(H)}\ge \abs{X_{i,j}}+\alpha_{i,j}(G)/4$.
  \end{CLAIM}

  Before proving Claim~\ref{claim1}, let us see why this claim implies
  our lemma.
  First we would like to see why there exist $i$ and $j$ such that
  $\delta_G(X_{i,j})$ has exactly two edges.
  If $p_1$ has degree bigger than $2$, then $p_1$ is adjacent to $q_2$
  and so $G*q_1=G\setminus p_1q_2$. 
  Thus we may assume that both $p_1$ and $q_1$ have degree $2$.
  Keep in mind that the number of vertices of degree $3$ or $4$ in
  $X_{1,1}$ may be decreased by $1$ by replacing $G$
  with $G*q_1$ and so $\alpha_{1,1}(G)\ge \alpha-2$.
  
  By applying Claim~\ref{claim1} with $i=j=1$, 
  we obtain a generalized ladder $H$ of maximum degree at most $3$ as a
  vertex-minor
  such that $\abs{V(H)}\ge 2+(\alpha-2)/4\ge \alpha/4$.
  This completes the proof of the lemma, assuming Claim~\ref{claim1}.

  \medskip

  We now prove Claim~\ref{claim1} by induction on $\abs{V(G)}-\abs{X_{i,j}(G)}$. 
  We may assume that every vertex in $V(G)\setminus (X_{i,j}\cup
  \{p_a,q_b\})$ has degree $3$ or $4$
  because otherwise we can apply local complementation and delete it
  while keeping $\alpha_{i,j}$.
  Then $p_{i+1}$ is obviously adjacent to $q_{j+1}$.

  We may assume that $i<a-1$ 
  because otherwise $G$ is a generalized ladder of maximum degree
  $3$ if $p_a$ has degree $3$
  and $G\setminus q_b$ is a generalized ladder of maximum degree $3$
  otherwise.
  Similarly we may assume $j<b-1$.
 Either $p_{i+1}$ or $q_{j+1}$  has degree $4$, because otherwise
 $\delta_G(X_{i+1,j+1})$  has exactly two edges.
 By symmetry, we may assume that $p_{i+1}$ has degree $3$ and 
 $q_{j+1}$ has degree $4$
 and therefore $q_{j+1}$ is adjacent to $p_{i+2}$.

  If $\alpha_{i,j}(G)\le 12$, then $H=G[X_{i+2,j+1}]$ is a generalized
  ladder of maximum degree at most $3$.
  Thus we may assume that $\alpha_{i,j}(G)>12$.
  If $b-j\le 4$, then $a-i\le 8$ because each vertex in
  $q_{j+1},q_{j+2},\ldots,q_b$ has degree at most $4$
  and each vertex in $p_{i+1},p_{i+2},\ldots,p_{a-1}$ has degree at
  least $3$.
  This contradicts our assumption that $\alpha_{i,j}(G)>12$. 
  So we may assume that $b-j\ge 5$ and similarly $a-i\ge 5$.

Let $R$ be the component of $G\setminus (E(P)\cup E(Q))$ containing
$p_{i+1}$. Because of the degree condition, $R$ is a path.
We now consider six cases, see Figure \ref{fig:removedeg4}.

\begin{figure}
  \centering
  \tikzstyle{v}=[circle, draw, solid, fill=black, inner sep=0pt, minimum width=3pt]
  \tikzstyle{vv}=[inner sep=0pt, minimum width=3pt]
  \tikzstyle{background rectangle}=[draw=blue!50, rounded corners=1ex]
  \tikzstyle{showbg}=[]
  \newcommand\setup[2]{
    \foreach \x in {0}{
        \node (p\x) [v,label=$p_{i}$] at (.3,.8) {};
        \node (q\x) [v,label=below:$q_{j}$] at (.3,0) {};
      }
    \foreach \x in {1,...,#1}      {
        \node (p\x) [v,label=$p_{i+\x}$] at (\x,.8) {};
        }
        \foreach \x in {1,...,#2} {
        \node (q\x) [v,label=below:$q_{j+\x}$] at (\x,0) {};
      }
      \draw [dotted] (p0)--(q0);
      \foreach \y in {0,.8}
      {
        \draw(-.2,\y) [dotted] -- (0,\y);
      }
      \draw(0,.8)--(#1.3,.8);
      \draw(0,0)--(#2.3,0);
      \draw [dashed,color=gray] (.5,-.5)--(.5,1.3);
      \node [left] at (.3,.3){$X_{i,j}$};
    }
  \newcommand\setuphide[7]{
    \foreach \x in {0}{
        \node (p\x) [v,label=$p_{i}$] at (.3,.8) {};
        \node (q\x) [v,label=below:$q_{j}$] at (.3,0) {};
      }
      \draw [dotted](p0)--(q0);
      \foreach \y in {0,.8}
      {
        \draw(-.2,\y) [dotted] -- (0,\y);
      }
      \draw(0,.8)--(#1.3,.8);
      \draw(0,0)--(#2.3,0);
      \node [left] at (.3,.3){$\Rightarrow #7$};
    \foreach \x in {#3}      {
        \node (p\x) [v,label=$p_{i+\x}$] at (\x,.8) {};
        }
        \foreach \x in {#5} {
        \node (q\x) [v,label=below:$q_{j+\x}$] at (\x,0) {};
      }
    \foreach \x in {#4}      {
        \node (p\x) [v,fill=red,label=$p_{i+\x}$] at (\x,.8) {};
        }
        \foreach \x in {#6} {
        \node (q\x) [v,fill=red,label=below:$q_{j+\x}$] at (\x,0) {};
      }
    }
\subfloat[(a) Apply $G*p_{i+2}\setminus p_{i+2}$]{\begin{tikzpicture}[showbg]    \setup32
    \draw (p1)--(q1)--(p2);\draw (p3)--(q2);
    \draw [dashed](q2)--(2.3,.1);
  \begin{scope}[xshift=6cm]
      \setuphide32{}{1,3}{2}{1}{X_{i+1,j}}
    \draw (p3)--(q2);
    \draw [dashed](q2)--(2.3,.1);
      \draw [dashed,color=gray] (.5,-.5)--(1.5,1.3);
      \draw [thick] (p1)--(p3)--(q1);
  \end{scope}
    \end{tikzpicture}
  }

  \subfloat[(b) Apply $G*p_{i+2}*q_{j+2}\setminus p_{i+2}\setminus q_{j+2}$]{\begin{tikzpicture}[showbg]      \setup33
    \draw (p1)--(q1)--(p2);\draw (p3)--(q2);\draw (p3)--(q3);
    \draw [dashed](q3)--(3.3,.1);
  \begin{scope}[xshift=6cm]
      \setuphide33{}{1,3}{}{1,3}{X_{i+1,j+1}}
    \draw [dashed](q3)--(3.3,.1);
      \draw [dashed,color=gray] (1.5,-.5)--(1.5,1.3);
      \draw [thick] (p1)--(p3);\draw [thick] (q3)--(q1);
  \end{scope}
    \end{tikzpicture}  
  }
  
  \subfloat[(c) Apply $G*q_{j+2}\setminus q_{j+2}$]{\begin{tikzpicture}[showbg]      \setup33
    \draw (p1)--(q1)--(p2)--(q2);\draw (p3)--(q3);
    \draw [dashed](p3)--(3.3,.7);
  \begin{scope}[xshift=6cm]
      \setuphide33{1,3}{2}{}{1,3}{X_{i+1,j+1}}
      \draw (p1)--(q1);
      \draw (p3)--(q3);
      \draw [dashed](p3)--(3.3,.7);
      \draw [dashed,color=gray] (1.5,-.5)--(1.5,1.3);
      \draw [thick] (p2)--(q3);\draw [thick] (q3)--(q1);
  \end{scope}
    \end{tikzpicture}    }

  \subfloat[(d) Apply $G*q_{j+2}*p_{i+3}\setminus q_{j+2}\setminus p_{i+3}$]{\begin{tikzpicture}[showbg]      \setup43
    \draw (p1)--(q1)--(p2)--(q2);\draw (p3)--(q3)--(p4);
    \draw [dashed](p4)--(4.3,.7);
  \begin{scope}[xshift=6cm]
      \setuphide43{1}{2,4}{}{1,3}{X_{i+2,j+1}}
      \draw (p1)--(q1);
    \draw [dashed](p4)--(4.3,.7);
      \draw [dashed,color=gray] (1.5,-.5)--(2.5,1.3);
      \draw [thick] (p2)--(p4);\draw [thick] (q3)--(q1);
  \end{scope}
     \end{tikzpicture}    }

  \subfloat[(e) Apply $G\pivot p_{i+2}q_{j+2}*p_{i+3}
  \setminus p_{i+2}\setminus q_{j+2}\setminus p_{i+3}$]{\begin{tikzpicture}[showbg]      \setup43
    \draw (p1)--(q1)--(p2)--(q2)--(p3);\draw (q3)--(p4);
    \draw [dashed](q3)--(3.3,.1);
    \draw [dashed](p4)--(4.3,.7);
  \begin{scope}[xshift=6cm]
      \setuphide43{}{1,4}{}{1,3}{X_{i+1,j+1}}
    \draw [dashed](p4)--(4.3,.7);
      \draw [dashed,color=gray] (1.5,-.5)--(1.5,1.3);
      \draw [thick] (p1)--(p4); 
      \draw [thick] (q1)--(q3);
    \draw [dashed](q3)--(3.3,.1);
  \end{scope}
    \end{tikzpicture}    }

  \subfloat[(f) Apply $G\pivot p_{i+2}q_{j+2}\setminus
  p_{i+2}\setminus q_{j+2}$]{\begin{tikzpicture}[showbg]      \setup33
    \draw (p1)--(q1)--(p2)--(q2)--(p3)--(q3);
    \draw [dashed](q3)--(3.3,.1);
  \begin{scope}[xshift=6cm]
      \setuphide33{}{1,3}{}{1,3}{X_{i,j+1}}
    \draw [dashed](q3)--(3.3,.1);
      \draw [dashed,color=gray] (1.5,-.5)--(.5,1.3);
      \draw [thick] (p3)--(p1)--(q3)--(q1);
  \end{scope}
    \end{tikzpicture}    }
  \caption{Cases in the proof of Lemma~\ref{lem:removedeg4}.}
  \label{fig:removedeg4}
\end{figure}
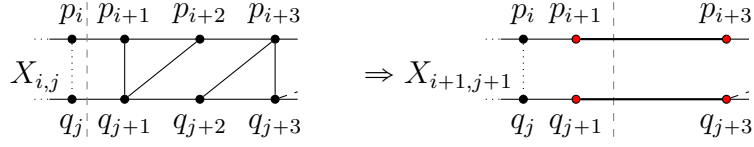
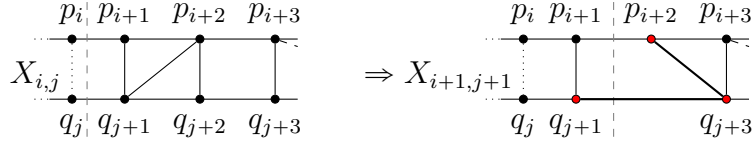
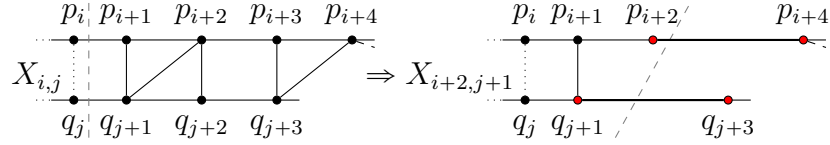
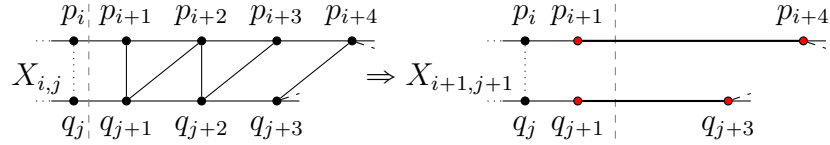
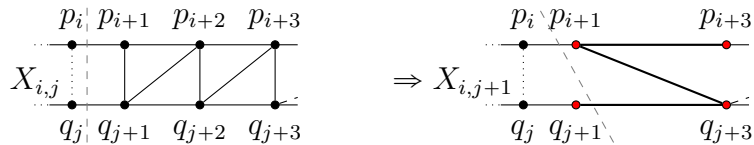

 \begin{enumerate}[(a)]
 \item %
   If $R$ has length $2$ and $p_{i+3}$ has degree $3$ in $G$, 
   then $G'=G*p_{i+2}\setminus p_{i+2}=
   (G\setminus p_{i+2} +p_{i+1}p_{i+3}+q_{j+1}p_{i+3})\setminus
   p_{i+1}q_{j+1}$
   is a generalized ladder of maximum degree at most $4$.
   Every vertex in $G'$ not in $X_{i,j}$ has degree at most $4$.
   Furthermore $p_{i+1}$ has degree $2$ in $G'$.
   Thus, $\delta_{G'}(X_{i+1,j})$ has exactly $2$ edges.
   Then $\abs{X_{i+1,j}}+\alpha_{i+1,j}(G')/4
   \ge (\abs{X_{i,j}}+1)+ (\alpha_{i,j}(G)-2)/4 \ge
   \abs{X_{i,j}}+\alpha_{i,j}(G)/4$.
   By the induction hypothesis, we find a desired vertex-minor $H$ in
   $G'$.
 \item %
   If $R$ has length $2$ and $p_{i+3}$ has degree $4$ in $G$, then
   the vertex $q_{j+2}$ has degree $3$.
   Then $G'=G*p_{i+2}*q_{j+2}\setminus p_{i+2}\setminus q_{j+2}$
   is a generalized ladder of maximum degree at most $4$.
   Then $\delta_{G'}(X_{i+1,j+1})$ has exactly two edges
   and $\alpha_{i+1,j+1}(G')\ge \alpha_{i,j}(G)-6$.
   Again, $\abs{X_{i+1,j+1}}+\alpha_{i+1,j+1}(G')/4\ge
   \abs{X_{i,j}}+2+(\alpha_{i,j}(G)-6)/4\ge
   \abs{X_{i,j}}+\alpha_{i,j}(G)/4$
   and therefore we are done.
 \item %
   If $R$ has length $3$ and $q_{j+3}$ has degree $3$ in $G$, then
   $G'=G*q_{j+2}\setminus q_{j+2}$ 
   is a generalized ladder of maximum degree at most $4$.
   Then $\delta_{G'}(X_{i+1,j+1})$ has exactly two edges
   and $\alpha_{i+1,j+1} (G')\ge \alpha_{i,j}(G)-3$.
   We deduce that $\abs{X_{i+1,j+1}}+\alpha_{i+1,j+1}(G')/4\ge
   \abs{X_{i,j}}+2+(\alpha_{i,j}(G)-3)/4
   \ge \abs{X_{i,j}}+\alpha_{i,j}(G)/4$.
 \item %
   If $R$ has length $3$ and $q_{j+3}$ has degree $4$ in $G$, then 
   $p_{i+3}$ has degree $3$
   and 
   $G'=G*q_{j+2}*p_{i+3}\setminus q_{j+2}\setminus p_{i+3}$ 
   is a generalized ladder of maximum degree at most $4$.
   Then $\delta_{G'}(X_{i+2,j+1})$ has exactly two edges
   and $\alpha_{i+2,j+1} (G')\ge \alpha_{i,j}(G)-7$.
   We deduce that $\abs{X_{i+2,j+1}}+\alpha_{i+2,j+1}(G')/4\ge
   \abs{X_{i,j}}+3+(\alpha_{i,j}(G)-7)/4\ge
   \abs{X_{i,j}}+\alpha_{i,j}(G)/4$.
   By the induction hypothesis, $G'$ has a desired vertex-minor and so
   does $G$.

   \item %
     If $R$ has length $4$, then
     $G'=G\pivot p_{i+2}q_{j+2}*p_{i+3}\setminus p_{i+2}\setminus p_{i+3}\setminus q_{j+2}$
     is a generalized ladder of maximum degree at most $4$.
     Then $\delta_{G'}(X_{i+1,j+1})$ has exactly two edges
     and $\alpha_{i+1,j+1}(G')\ge \alpha_{i,j}(G)-7$
     and therefore $\abs{X_{i+1,j+1}}+\alpha_{i+1,j+1}(G')/4
     \ge \abs{X_{i,j}}+2+(\alpha_{i,j}(G)-7)/4
     \ge \abs{X_{i,j}}+\alpha_{i,j}(G)/4$.
     Our induction hypothesis implies that $G'$ has a desired
     vertex-minor.

   \item %
     If $R$ has length at least $5$, 
     then      $G'=G\pivot p_{i+2}q_{j+2}\setminus p_{i+2}\setminus q_{j+2}$
     is a generalized ladder of maximum degree at most $4$.
     Then $\delta_{G'}(X_{i,j+1})$ has exactly two edges
     and $\alpha_{i,j+1}(G')\ge \alpha_{i,j}(G)-4$
     and therefore $\abs{X_{i,j+1}}+\alpha_{i,j+1}(G')/4
     \ge\abs{X_{i,j}}+1+(\alpha_{i,j}(G)-4)/4
     = \abs{X_{i,j}}+\alpha_{i,j}(G)/4$.
     Our induction hypothesis implies that $G'$ has a desired
     vertex-minor.
 \end{enumerate}
  In all cases, we find the
  desired
  vertex-minor $H$.
  This completes the proof of Claim~\ref{claim1}.
\end{proof}

\begin{LEM}\label{lem:ladder-final2}
  Let $G$ be a generalized ladder of maximum degree at most $4$.
  If $\abs{V(G)}\ge 192n^3$, then $G$ has a cycle of length $4n+3$ as a vertex-minor.
\end{LEM}
\begin{proof}
  Let $P$, $Q$ be two defining paths of $G$. 
  We may assume $a>1$ and $b>1$ because $(192n^3-2)/3+2\ge 4n+3$.

  Let $p_xq_y$ be the unique chord other than $p_1q_1$ with minimum $x+y$.
  We claim that we may assume $(x-1)+(y-1)\le 2$.
  Suppose not. Then $p_xq_y$, $p_1q_1$ and subpaths of $P$ and $Q$
  form a cycle of length $x+y\ge 5$ and
  $p_1,p_2,\ldots,p_{x-1},q_1,q_2,\ldots,q_{y-1}$ have degree $2$.
  By moving the first few vertices of $P$ to $Q$ or $Q$ to $P$,
  we may assume that $x\ge 3$ and $y\ge 2$.
  Then we may replace $G$ with $G*p_1$. This proves the claim.

  Thus the induced cycle containing $p_1q_1$ has at most $2$ edges
  from $E(P)\cup E(Q)$.
  Similarly we may assume that 
  the induced cycle containing $p_aq_b$ has at most $2$ edges
  from $E(P)\cup E(Q)$.

  If $G$ has at least $48n^2$ vertices of degree $3$ or $4$, then 
  by Lemma~\ref{lem:removedeg4}, $G$ has a generalized ladder $H$ as a
  vertex-minor such that $\abs{V(H)}\ge 12n^2$ and $H$ has maximum degree
  at most $3$. By Lemma~\ref{lem:ladder-final}, 
  $H$ has a cycle of length $4n+3$ as a vertex-minor.

  Thus we may assume that $G$ has less than $48n^2$ vertices of degree
  $3$ or $4$. 
  We may assume that $G$ has at least one vertex of degree at least $3$.
  The cycle formed by edges in $E(P)\cup E(Q)\cup
  \{p_1q_1,p_aq_b\}$ is partitioned into less than $48n^2$ paths whose
  internal vertices have degree $2$ in $G$.
  One of the paths has length greater than $192n^3/(48n^2)=4n$.
  Then there is an induced cycle $C$ of $G$ containing this path.
  Since $C$ does not contain $p_1q_1$ or $p_aq_b$, $C$ must contain
  two edges not in $E(P)\cup E(Q)\cup \{p_1q_1,p_aq_b\}$.
  Thus the length of $C$ is at least $4n+3$.
\end{proof}
\subsection{Treating all generalized ladders}
\begin{LEM}\label{lem:boundingmaxdeg}
  Let $G$ be a generalized ladder.
  If $G$ has $n$ vertices of degree at least $4$, 
  then $G$ has a vertex-minor $H$ 
  that is a generalized ladder such that 
  the maximum degree of $H$ is at most $4$
  and $H$ has at least $n$ vertices.
\end{LEM}
\begin{proof}
  Let $S$ be the set of vertices having degree at least $4$.
  For each vertex $v$ in $S$, let $P_v$ be the minimal subpath of 
  $Q$ containing all neighbors of $v$ in $Q$ if $v\in V(P)$
  and let $P_v$ be the minimal subpath
  of $P$ containing all neighbors of $v$ in $P$ if $v\in V(Q)$.

  Then each internal vertex of $P_v$ has degree $2$ or $3$
  and has degree $3$ if and only if it is adjacent to $v$.
  We apply 
  local complementation to each internal vertex and delete all
  internal vertices of $P_v$.
  It is easy to see that the resulting graph $H$  is a generalized ladder
  and moreover $S\subseteq V(H)$ and 
  every vertex in $S$ has degree at most $4$ in $H$.
\end{proof}

We are now ready to prove the main proposition of this section.
\begin{proof}[Proof of Proposition~\ref{prop:ladder}]
  Let $G$ be such a graph.
  If $G$ has at least $192n^3$ vertices of degree at least $4$, 
  then by Lemma~\ref{lem:boundingmaxdeg}, $G$ has a vertex-minor $H$
  having at least $192n^3$ vertices such that
  $H$ is a generalized ladder of maximum degree at most $4$.
  By Lemma~\ref{lem:ladder-final2}, $H$ has a cycle of length $4n+3$
  as a vertex-minor.

  Thus we may assume that $G$ has less than  $192n^3$ vertices of degree
  at least $4$. 
  For a vertex $v$ in $P$ having degree at least $5$, 
  let $q_i$, $q_j$ be two neighbors of $v$ in $Q$ such that 
  if $q_k$ is a neighbor of $v$ in $Q$, then $i\le k\le j$.
  By Lemma~\ref{lem:pathone}, if $j-i+2\ge 24n^2-3$, then $G$ contains a
  cycle of length $4n+3$ as a vertex-minor. Thus we may assume
  $j-i\le 24n^2-6$.
  The subpath of $Q$ from $q_i$ to $q_j$  contains $j-i-1\le 24n^2-7$
  internal vertices.
  Similarly the same bound holds for a vertex $v$ in $Q$ having degree
  at least $5$.
  As in the proof of Lemma~\ref{lem:boundingmaxdeg}, 
  we apply local complementation and delete all internal vertices 
  of the minimal path spanning the neighbors of 
  each vertex of degree at least $5$
  to
  obtain $H$.
  Then each vertex of degree at least $5$  in $G$ will have degree at most
  $4$ in $H$.
  Since we remove at most $(192n^3-1)(24n^2-7)$ vertices, 
  \[   \abs{V(H)} \ge \abs{V(G)}-(192n^3-1)(24n^2-7)
  >192n^3.\]
  By Lemma~\ref{lem:ladder-final2}, $H$ has  a cycle of length $4n+3$
  as a vertex-minor.
\end{proof}
\section{Blocking sequences}\label{sec:blocking}
Let $A,B$ be two disjoint subsets of the vertex set of a graph $G$.
By the definition of $\rho^*_G$ and $\rho_G$,
it is clear that \[
\text{if }A\subseteq X\subseteq V(G)\setminus B, 
\text{ then }
\rho^*_G(A,B)\le \rho_G(X).\]
What prevents us to achieve the equality for some $X$?
We now present a tool called a blocking sequence, 
that is a certificate to guarantee that 
no such $X$ exists.
Blocking sequences were introduced by Geelen~\cite{Geelen1995}.

A sequence $v_1,v_2,\ldots,v_m$ ($m\ge 1$) 
is called  a \emph{blocking sequence}
of a pair $(A,B)$ of disjoint subsets $A$, $B$ of $V(G)$
if 
\begin{enumerate}[(a)]
\item\label{b1} $\rho^*_G(A,B\cup\{v_1\})>\rho^*_G(A,B)$,
\item\label{b2} $\rho^*_G(A\cup \{v_i\},B\cup\{v_{i+1}\})>\rho^*_G(A,B)$
  for all $i=1,2,\ldots, m-1$,
\item\label{b3} $\rho^*_G(A\cup \{v_m\},B)>\rho^*_G(A,B)$,
\item no proper subsequence of $v_1,\ldots,v_m$ satisfies (a), (b),
  and (c).
\end{enumerate}

The condition (d) is essential for the following standard lemma. 
\begin{LEM}\label{lem:cut-blocking}
  Let $v_1,v_2,\ldots,v_m$ be a blocking sequence for $(A,B)$ in a
  graph $G$. 
  Let $X$, $Y$ be disjoint subsets of $\{v_1,v_2,\ldots,v_m\}$
  such that if $v_i\in X$ and $v_j\in Y$, then $i<j$.
  Then 
  \[\rho^*_G(A\cup X,B\cup Y)=\rho^*_G(A,B)\]
  if and only if
  $v_1\notin Y$, $v_m\notin X$,
  and for all $i\in\{1,2,\ldots,m-1\}$, 
  either $v_i\notin X$ or $v_{i+1}\notin Y$.
\end{LEM}
\begin{proof}
  The forward direction is trivial.
  Let us prove the backward implication. 
  Let $k=\rho^*_G(A,B)$.
  It is enough to prove $\rho^*_G(A\cup X,B\cup Y)\le k$.
  Suppose that 
  $v_1\notin Y$, $v_m\notin X$,
  and for all $i\in\{1,2,\ldots,m-1\}$, 
  either $v_i\notin X$ or $v_{i+1}\notin Y$
  and yet
  $\rho^*_G(A\cup X,B\cup Y)>k$.
  We may assume that $\abs{X}+\abs{Y}$ is chosen to be minimum.
  If $\abs{X}\ge 2$, then we can partition $X$ into two nonempty sets
  $X_1$ and $X_2$.
  Then by the hypothesis,
  $\rho^*_G(A\cup X_1,B\cup Y)=\rho^*_G(A\cup X_2,B\cup
  Y)=k$.
  By Lemma~\ref{lem:submodular}, we deduce
  that $\rho^*_G(A\cup X_1,B\cup Y)+\rho^*_G(A\cup X_2,B\cup Y)
  \ge k+\rho^*_G(A\cup X,B\cup Y)$
  and therefore we deduce that $\rho^*_G(A\cup X,B\cup Y)\le k$.
  So we may assume $\abs{X}\le 1$. By symmetry we may also assume
  $\abs{Y}\le 1$.
  Then by the condition (d), this is clear. 
\end{proof}

The following proposition states that a blocking sequence is a
certificate that $\rho_G(X)>\rho^*_G(A,B)$ for all $A\subseteq
X\subseteq V(G)\setminus B$.
This appears in almost all applications of blocking sequences. The
proof uses the submodular inequality (Lemma \ref{lem:submodular}).
\begin{PROP}[{Geelen~\cite[Lemma 5.1]{Geelen1995}; see
    Oum~\cite{Oum2006}}]
  \label{prop:block}
  Let $G$ be a graph and $A$, $B$ be two disjoint subsets of
  $V(G)$. Then $G$ has a blocking sequence for $(A,B)$
  if and only if 
  $\rho_G(X)>\rho^*_G(A,B)$ 
  for all $A\subseteq X\subseteq V(G)\setminus B$.
\end{PROP}
The following proposition allows us to change the graph 
to reduce the length of a blocking sequence.
This was pointed out by Geelen
[private communication with the second author, 2005].
A special case of the following proposition is presented in \cite{Oum2006}.
\begin{PROP}\label{prop:reduce-blocking}
  Let $G$ be a graph and $A$, $B$ be disjoint subsets of $V(G)$. 
  Let $v_1,v_2,\ldots,v_m$ be a blocking sequence for $(A,B)$ in $G$. 
  Let $1\le i\le m$.
  \begin{itemize}
  \item   If $m>1$, then 
  $\rho_{G*v_i}^*(A,B)=\rho_G^*(A,B)$ and 
  a sequence \[v_1,v_2,\ldots,v_{i-1},v_{i+1},\ldots,v_m\] obtained by
  removing $v_i$ from
  the blocking sequence is a blocking
  sequence  for $(A,B)$
  in $G*v_i$.
  \item If $m=1$, then
  $\rho_{G*v_i}^*(A,B)=\rho_G^*(A,B)+1$.
  \end{itemize}
\end{PROP}
\begin{proof}
  Let $k=\rho^*_G(A,B)$ and $H=G*v_i$.

  If $m=1$, then by Lemma~\ref{lem:reduce-submodular-2}, 
  \[
  \rho_H^*(A,B)+\rho_G^*(A,B)\ge
  \rho_G^*(A\cup\{v_1\},B)+\rho_G^*(A,B\cup\{v_1\})-1
  \ge 2k+1\]
  and therefore $\rho_H^*(A,B)\ge k+1$. 
  Since $\rho_H^*(A,B)\le \rho_H^*(A,B\cup\{v_1\})=\rho_G^*(A,B\cup\{v_1\})\le k+1$, 
  we deduce that $\rho_H^*(A,B)=k+1$ if $m=1$.

  Now we assume that $m\neq 1$.
  First it is easy to observe that 
  $\rho^*_{H}(X,Y)\le  \rho_{G}^*(X,Y\cup \{v_i\})$
  and $\rho^*_{H}(X,Y)\le  \rho_{G}^*(X\cup\{v_i\},Y)$
  whenever $X$, $Y$ are disjoint subsets of
  $V(G)\setminus\{v_i\}$,
  because the local complementation does not change the cut-rank
  function of $G[X\cup Y\cup\{v_i\}]$.
  This with Lemma~\ref{lem:cut-blocking} implies that
  \begin{itemize}
  \item 
    $\rho^*_{H}(A,B)\le k$, 
  \item
    $\rho^*_{H} (A\cup\{v_j\},B)\le k$ for all $j\in
    \{1,2,\ldots,m\}\setminus\{i-1,m\}$,
  \item 
    $\rho^*_H (A\cup\{v_{i-1}\}, B)\le k$ if $i\neq 1,m$.
  \item
    $\rho^*_{H} (A,B\cup\{v_j\})\le k$ for all $j\in
    \{1,2,\ldots,m\}\setminus\{1,i+1\}$.
  \item 
    $\rho^*_{H}(A,B\cup\{v_{i+1}\})\le k$ if $i\neq 1,m$.
  \item
    $\rho^*_{H}(A\cup\{v_j\},B\cup\{v_\ell\})\le k$
    for all  $j,\ell\in \{1,2,\ldots,m\}\setminus\{i\}$ with
    $\ell-j>1$,
    unless $j+1=i=\ell-1$.
  \end{itemize}
  Let $B'=B\cup \{v_{i+1}\}$ if $i<m$ and $B'=B$ otherwise.
  Then
  $\rho^*_G(A\cup\{v_i\},B')=k+1$
  and $\rho^*_G(A,B')=k$.

  (1)
  We claim that if $i>1$, then $\rho_H^*(A,B\cup\{v_1\})>k$.
  By Lemma~\ref{lem:reduce-submodular-2}, 
  \[
  \rho_{H}^*(A,B'\cup\{v_1\})+\rho_G^*(A,B')
  \ge \rho_G^*(A,B'\cup\{v_1,v_i\})+\rho_G^*(A\cup\{v_i\},B')-1,\]
  and therefore we deduce that
  $\rho_H^*(A,B'\cup\{v_1\})\ge 
  \rho_G^*(A,B'\cup\{v_1,v_i\})>k$.
  By Lemma~\ref{lem:submodular}, 
  $\rho_H^*(A,B'\cup\{v_i\})+\rho_H^*(A,B\cup\{v_1\})\ge 
  \rho_H^*(A,B'\cup \{v_1,v_i\})+\rho_H^*(A,B)>2k$.
  We deduce that $\rho_H^*(A,B\cup\{v_1\})>k$
  because $\rho_H^*(A,B'\cup\{v_i\})=\rho_G^*(A,B'\cup\{v_i\})=k$ by
  Lemma \ref{lem:cut-blocking}.

  (2) 
  By (1) and symmetry between $A$ and $B$, 
  if $i<m$, then 
  $\rho_H^*(A\cup\{v_m\},B)>k$.

  Then we deduce that $\rho_H^*(A,B)\ge k$ and therefore 
  $\rho_H^*(A,B)=k$.

  (3) We claim that if $j<i-1$, then 
  $\rho_H^*(A\cup\{v_j\},B\cup\{v_{j+1}\})>k$.
  By Lemma~\ref{lem:reduce-submodular-2}, 
  \begin{multline*}
    \rho_{H}^*(A\cup\{v_j\},B'\cup\{v_{j+1}\})+\rho_G^*(A\cup\{v_j\},B')\\
    \ge
    \rho_G^*(A\cup\{v_j\},B'\cup\{v_{j+1},v_i\})+\rho_G^*(A\cup\{v_j,v_i\},B')-1
    >2k,
  \end{multline*}
  and therefore $\rho_H^*(A\cup\{v_j\},B'\cup\{v_{j+1}\})>k$.
  By Lemma~\ref{lem:submodular}, 
  $\rho_H^*(A\cup\{v_j\},B\cup\{v_{j+1}\})
  +\rho_H^*(A\cup\{v_j\},B')
  \ge \rho_H^*(A\cup\{v_j\},B'\cup\{v_{j+1}\})+
  \rho_H^*(A\cup\{v_j\},B)>2k$.
  Note that $\rho_H^*(A\cup \{v_j\},B)\ge \rho_H^*(A,B)=k$.
  Since $\rho_H^*(A\cup\{v_j\},B')\le
  \rho_H^*(A\cup\{v_j\},B'\cup\{v_i\})
  =\rho_G^*(A\cup\{v_j\},B'\cup\{v_i\})
  \le k$, we deduce that 
  $\rho_H^*(A\cup\{v_j\},B\cup\{v_{j+1}\})>k$.

  (4) By symmetry, we deduce from (3) that 
  if $i<j<m$, then 
  $\rho_H^*(A\cup\{v_j\},B\cup\{v_{j+1}\})>k$.

  (5) We claim that 
  $\rho_H^*(A\cup\{v_{i-1}\},B')>k$.
  By Lemma~\ref{lem:reduce-submodular-2},
  \begin{multline*}
  \rho_H^*(A\cup\{v_{i-1}\},B')+
  \rho_G^*(A\cup\{v_{i-1}\},B')\\
  \ge 
  \rho_G^*(A\cup\{v_{i-1}\},B'\cup\{v_i\})
  +\rho_G^*(A\cup\{v_{i-1},v_i\},B')-1>2k.
  \end{multline*}
  Since $\rho_G^*(A\cup\{v_{i-1}\},B')=k$, we have
  $\rho_H^*(A\cup\{v_{i-1}\},B')>k$.

  This completes the proof of the lemma that
  $v_1,v_2,\ldots,v_{i-1},v_{i+1},\ldots,v_m$ is a blocking sequence
  of $(A,B)$ in $G*v_i$.
\end{proof}

\begin{COR}\label{cor:blocking-pivot}
  Let $G$ be a graph and $A$, $B$ be disjoint subsets of $V(G)$. 
  Let $v_1,v_2,\ldots,v_m$ be a blocking sequence for $(A,B)$ in $G$. 
  Let $1\le i\le m$.
  Suppose that $v_i$ has a neighbor $w$ in $A\cup B$.
  \begin{itemize}
  \item   If $m>1$, then 
  $\rho_{G\pivot v_iw}^*(A,B)=\rho_G^*(A,B)$ and 
  the sequence 
  $v_1,v_2,\ldots,v_{i-1},v_{i+1},\ldots,v_m$ obtained by removing $v_i$ from
  the blocking sequence is a blocking
  sequence  for $(A,B)$
  in $G\pivot v_iw$.
  \item If $m=1$, then
  $\rho_{G\pivot v_iw}^*(A,B)=\rho_G^*(A,B)+1$.
  \end{itemize}
\end{COR}
\begin{proof}
  It follows easily from the facts that $G\pivot v_i w= G*w *v_i * w$ and 
  $\rho_G^*(X,Y)=\rho_{G*w}^*(X,Y)$ for all graphs $G$ with $w\in X\cup Y$.
\end{proof}

\begin{COR}\label{cor:blocking-pivot2}
  Let $G$ be a graph and $A$, $B$ be disjoint subsets of $V(G)$. 
  Let $v_1,v_2,\ldots,v_m$ be a blocking sequence for $(A,B)$ in $G$. 
  Let $1\le i\le m$.
  Suppose that $v_i$ and $v_{i'}$ are adjacent and $i<i'$.
  \begin{itemize}
  \item   If $m>2$, then 
  $\rho_{G\pivot v_iv_{i'}}^*(A,B)=\rho_G^*(A,B)$ and 
  the sequence 
  $v_1,v_2,\ldots,v_{i-1},v_{i+1},\ldots,v_{i'-1},v_{i'+1},\ldots,v_m$ 
  obtained by removing $v_i$ and $v_{i'}$ from
  the blocking sequence is a blocking
  sequence  for $(A,B)$
  in $G\pivot v_iv_{i'}$.
  \item If $m=2$, then
  $\rho_{G\pivot v_iv_{i'}}^*(A,B)=\rho_G^*(A,B)+1$.
  \end{itemize}
\end{COR}
\begin{proof}
  If $v_i$ has a neighbor $w$ in $A\cup B$, then 
  $G\pivot v_i v_{i'}=G\pivot v_i w\pivot wv_{i'}$ and this corollary follows
  from Corollary \ref{cor:blocking-pivot}.
  So we may assume that $v_i$ has no neighbors in $A\cup B$
  and similarly $v_{i'}$ has no neighbors in $A\cup B$.
  Thus $i,i'\notin \{ 1,m\}$ and $m\ge 4$.

  Since $v_i$ and $v_{i'}$ are adjacent, we may assume that $i'=i+1$.
  Let $H=G\pivot v_iv_{i+1}$
  and $k=\rho_G^*(A,B)$.
  Since $v_i$ and $v_{i+1}$ have no neighbors in $A\cup B$,
  $\rho_H^*(A,B)=k$.

  Then $v_1,v_2,\ldots,v_i$ is a blocking sequence for
  $(A,B\cup\{v_{i+1}\})$ in $G$ 
  by Lemma~\ref{lem:cut-blocking}.
  Similarly $v_{i+1},v_{i+2}, \ldots, v_m$ is a blocking sequence for
  $(A\cup\{v_i\},B)$ in $G$.

  By Corollary~\ref{cor:blocking-pivot}, 
  $v_1,v_2,\ldots,v_{i-1}$ is a blocking sequence for
  $(A,B\cup\{v_{i+1}\})$ in $H$.
  Then $\rho_H^*(A,B\cup\{v_1\})=\rho_H^*(A,B\cup\{v_1,v_{i+1}\})>k$,
  because $v_{i+1}$ has no neighbors of $H$ in $A$.

  For $1\le j<i-1$, 
  $\rho_H^*(A\cup\{v_j\},B\cup\{v_{j+1}\})
  +\rho_H^*(A\cup\{v_j\},B\cup\{v_{i+1}\})
  \ge \rho_H^*(A\cup\{v_j\},B\cup\{v_{j+1},v_{i+1}\})
  +\rho_H^*(A\cup\{v_j\},B)>2k$ and therefore
  \[\rho_H^*(A\cup\{v_j\},B\cup\{v_{j+1}\})>k\]
  because $\rho_H^*(A\cup \{v_j\},B)\le \rho_H^*(A\cup \{v_j\},B\cup \{v_{i+1}\})\le k$.

  Similarly $v_{i+2},v_{i+3},\ldots,v_m$ is a blocking sequence for
  $(A\cup\{v_i\},B)$ in $H$. 
  By symmetry, we deduce that 
  $\rho^*_H(A\cup \{v_m\},B)>k$
  and 
  $\rho_H^*(A\cup\{v_j\},B\cup \{v_{j+1}\})>k$ for all $i+1<j<m$.

  We now claim that $\rho_H^*(A\cup\{v_{i-1}\},B\cup\{v_{i+2}\})>k$.
  By Lemma~\ref{lem:submodular}, 
  \begin{multline*}
    \rho_H^*(A\cup\{v_{i-1}\},B\cup\{v_{i+2}\})
    +\rho_H^*(A\cup\{v_{i+1}\},B\cup\{v_{i+2}\})\\
    \ge 
    \rho_H^*(A\cup\{v_{i-1},v_{i+1}\},B\cup\{v_{i+2}\})
    +\rho_H^*(A,B\cup\{v_{i+2}\}).
  \end{multline*}
  Since $v_{i+1}$ has no neighbors in $A\cup B$,
  we have
  $\rho_H^*(A\cup\{v_{i+1}\},B\cup\{v_{i+2}\})
  =\rho_G^*(A\cup\{v_{i}\},B\cup\{v_{i+2}\})=k$
  and 
  $\rho_H^*(A,B\cup\{v_{i+2}\})=
  \rho_G^*(A,B\cup\{v_{i+2}\})=k$. Therefore
  \[
    \rho_H^*(A\cup\{v_{i-1}\},B\cup\{v_{i+2}\})
    \ge 
    \rho_H^*(A\cup\{v_{i-1},v_{i+1}\},B\cup\{v_{i+2}\}).
  \]
  By Lemma~\ref{lem:reduce-submodular-2}, 
  \begin{multline*}
    \rho_H^*(A\cup\{v_{i-1},v_{i+1}\},B\cup\{v_{i+2}\})
    +\rho_G^*(A\cup\{v_{i-1}\},B\cup\{v_{i+1},v_{i+2}\})
    \\\ge 
    \rho_G^*(A\cup\{v_{i-1}\},B\cup\{v_i,v_{i+1},v_{i+2}\})
    \\
    +
    \rho_G^*(A\cup\{v_{i-1},v_i,v_{i+1}\},B\cup\{v_{i+2}\})-1.
  \end{multline*}
  By Lemma~\ref{lem:cut-blocking}, 
  $\rho_G^*(A\cup\{v_{i-1},v_i,v_{i+1}\},B\cup\{v_{i+2}\})>k$
  and 
  $\rho_G^*(A\cup\{v_{i-1}\},B\cup\{v_{i+1},v_{i+2}\})=k$.
  Therefore
  $    \rho_H^*(A\cup\{v_{i-1}\},B\cup\{v_{i+2}\})
  \ge 
  \rho_H^*(A\cup\{v_{i-1},v_{i+1}\},B\cup\{v_{i+2}\})\ge 
  \rho_G^*(A\cup\{v_{i-1}\},B\cup\{v_i,v_{i+1},v_{i+2}\})>k$.
  This proves the claim.
  
  So far we have shown that the sequence
  $v_1,v_2,\ldots,v_{i-1},v_{i+2},\ldots,v_m$ satisfies (a), (b), (c)
  of the definition of blocking sequences. 
  It remains to show (d).
  For $j\in \{2,3,\ldots,m\}\setminus\{i,i+1\}$, 
  $\rho_H^*(A,B\cup\{v_j\})=\rho_G^*(A,B\cup\{v_j\})=k$
  because $v_i$
  and $v_{i+1}$ have no neighbors in $A\cup B$.
  Similarly 
  $\rho_H^*(A\cup\{v_j\},B)=\rho_G^*(A\cup\{v_j\},B)=k$
  for $j\in\{1,2,\ldots,m-1\}\setminus \{i,i+1\}$.
  For $j,\ell\in \{1,2,\ldots,m\}\setminus\{i,i+1\}$ with $\ell-j>1$, 
  either
  $\rho_G^*(A\cup\{v_j\},B\cup\{v_\ell,v_i,v_{i+1}\})= k$
  or 
  $\rho_G^*(A\cup\{v_j,v_i,v_{i+1}\},B\cup\{v_\ell\})= k$
  and therefore
  $\rho_H^*(A\cup\{v_j\},B\cup\{v_\ell\})  \le k$,
  unless $j=i-1$ and $\ell=i+2$.
  This  completes the proof.
\end{proof}

We will now prove that 
without loss of generality,
a blocking
sequence for $(A,B)$ is short
by applying local complementation  while keeping the subgraph induced on $A\cup B$.

\begin{PROP}\label{prop:short-patch}
  Let $G$ be a prime graph and 
  let $A$, $B$ be disjoint subsets of $V(G)$ with $\abs{A},\abs{B}\ge 2$.
  Suppose that there exist two nonempty sets $A_0\subseteq A$ and $B_0\subseteq B$ 
  such that the set of all edges between $A$ and $B$ is
  $\{xy:x\in A_0,y\in B_0\}$.
  Let 
  \[
  \ell_0=
  \begin{cases}
    3 & \text{if }\abs{A_0}=\abs{B_0}=1,\\
    4 & \text{if }\abs{A_0}=1 \text{ or } \abs{B_0}=1,\\
    6 &\text{otherwise.}
  \end{cases}
  \]
  Then there exists a graph $G'$ locally equivalent to $G$
  satisfying the following.
  \begin{enumerate}[(i)]
  \item $G[A\cup B]=G'[A\cup B]$.
  \item 
    $G'$ has a blocking sequence $b_1,b_2,\ldots,b_\ell$ of length at most $\ell_0$
    for $(A,B)$.
  \end{enumerate}
\end{PROP}
\begin{proof}
  Since $G$ is prime, $G$ has a blocking sequence for $(A,B)$ by Proposition~\ref{prop:block}.
  Let $\mathcal G$ be the set of all graphs $G'$ locally equivalent to $G$
  such that $G'[A\cup B]=G[A\cup B]$.
  We assume that $G$ is chosen in $\mathcal G$ 
  so that the length $\ell$ of a blocking
  sequence $b_1,b_2,\ldots,b_\ell$ for $(A,B)$ is minimized.

  For $1\le i<\ell$, $N_G(b_i)\cap B= B_0$ or $\emptyset$
  because $\rho_G(A\cup \{b_i\},B)=1$.
  For $1<i\le \ell$, $N_G(b_i)\cap A=A_0$ or $\emptyset$
  because $\rho_G(A, B\cup \{b_i\})=1$.

  Suppose that $N_G(b_i)\cap (A\cup B)=N_G(b_{j})\cap (A\cup B)$ for some
  $1<i<j<\ell$.
  If $b_i$ and $b_j$ are adjacent, then $G'=G\pivot b_i b_j\in \mathcal
  G$.
  If $b_i$ and $b_j$ are non-adjacent, then $G'=G*b_i *b_j\in \mathcal
  G$.
  In both cases, we found a graph in $\mathcal G$ having  a shorter
  blocking sequence by Proposition \ref{prop:reduce-blocking} or
  Corollary \ref{cor:blocking-pivot2}, contradicting our assumption.

  If $\abs{B_0}=1$, then for all $1<i<\ell$, $N_G(b_i)\cap A=A_0$
  because otherwise $G*b_i\in \mathcal G$ has a shorter blocking
  sequence by Proposition~\ref{prop:reduce-blocking}, contradicting
  our assumption.
  Similarly if $\abs{A_0}=1$, then $N_G(b_i)\cap B=B_0$ for all $1<i<\ell$.

  By the pigeonhole principle, we deduce that $\ell\le \ell_0$.
\end{proof}
\section{Obtaining a long cycle from a huge induced path}\label{sec:path}
In this section we aim to prove the following theorem.
\begin{THM}\label{thm:path-to-cycle}
  If a prime graph has an induced path of length 
  $\lceil 6.75 n^7\rceil$, then 
  it has a cycle of length $n$ as a vertex-minor.
\end{THM}
The main idea is to find a big generalized ladder, defined in
Section~\ref{sec:ladder} 
as  a vertex-minor
by using blocking sequences in Section~\ref{sec:blocking}.

\subsection{Patching a path}
For $1\le k\le n-2$, a \emph{$k$-patch} of an induced path
$P=v_0v_1\cdots v_n$ of a graph $G$
is a sequence $Q=w_1,w_2,\ldots,w_k$ of distinct vertices not on $P$
such that 
for each $i\in \{1,2,\ldots,k\}$, 
\begin{enumerate}[(i)]
\item $v_{i+2}$ is the only vertex adjacent to $w_i$
  among $v_{i+1}$, $v_{i+2}$, $\ldots$,  $v_{n}$,
\item $\emptyset\neq N_G(w_i)\cap 
  \{v_0,\ldots,v_i,w_1,\ldots,w_{i-1}\}
  \neq
    \{v_i,w_{i-1}\}$ if $i>1$,
  \item $N_G(w_1)\cap \{v_0,v_1\}=\{v_0\}$.
\end{enumerate}
An induced path is called \emph{$k$-patched} if it has a $k$-patch.
An induced path of length $n$ is called \emph{fully patched}
if it is equipped with a $(n-2)$-patch.
See Figure~\ref{fig:patched} for an example.
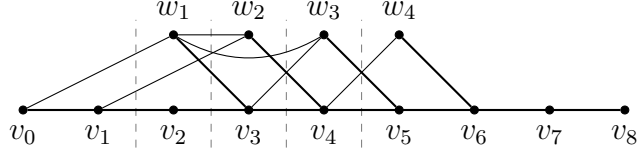
\begin{figure}
  \centering
  \tikzstyle{v}=[circle, draw, solid, fill=black, inner sep=0pt, minimum width=3pt]
  \begin{tikzpicture}
    \foreach \x in {0,...,8}
    {
      \node (v\x) at (\x,0) [v,label=below:$v_\x$] {};
    }
    \foreach \x in {1,...,4}
    {
      \node (w\x) at (\x+1,1) [v,label=$w_\x$] {};
    \draw [color=gray,dashed](\x+.5,-.5)--(\x+.5,1.2);
    }
    \draw [thick](v0)--(v8);
    \draw [thick](w1)--(v3);
    \draw [thick](w2)--(v4);
    \draw [thick](w3)--(v5);
    \draw [thick](w4)--(v6);
    \draw (w1)--(v0);
    \draw (w2)to (w1);
    \draw (w2)to (v1);
    \draw (w3)[bend left]to (w1);
    \draw (w3)to (v3);
    \draw (w4)to (v4);
  \end{tikzpicture}
  \caption{An example of a $4$-patched path of length $8$.}
  \label{fig:patched}
\end{figure}

Our goal is to find a fully patched long induced path  in a vertex-minor 
of a prime graph having a very long induced path.

\begin{LEM}\label{lem:shortening}
  Let $P=v_0v_1\ldots v_m$ be an induced path from $s=v_0$ to $t=v_m$ in a graph $G$
  and let $H$ be a connected induced subgraph of $G\setminus V(P)$.
  Let $v$ be a vertex in $V(G)\setminus (V(H)\cup V(P))$.
  Suppose that $N_G(V(H))\cap V(P)=\{s\}$, $\abs{E(P)}\ge 6(n-1)^2-5$, 
  and $v$ has neighbors in both $V(P)\setminus\{s\}$
  and $V(H)$.

  If $G$ has no cycle of length $2n+1$ as a vertex-minor,
  then there exist a graph $G'$ locally equivalent to $G$
  and an induced path $P'$ from $s$ to $t$ of $G'$ disjoint from $V(H)$
  satisfying
  the following.
  \begin{enumerate}[(i)]
  \item $G[V(H)\cup\{s\}]=G'[V(H)\cup\{s\}]$,
  \item $N_{G}(v)\cap V(H)=N_{G'}(v)\cap V(H)$,
  \item $P'=v_0 v_iv_{i+1}v_{i+2}\cdots v_m$ for some $i$,
  \item $v_i$ is the only vertex on $V(P')$ adjacent to $v$ in $G'$,
  \item $\abs{E(P')}\ge \abs{E(P)}-6(n-1)^2+6$.
  \end{enumerate}
\end{LEM}
\begin{proof}
  Since $G$ has a cycle using $H$ with $s$ and $P$, $G$ is not a
  forest and therefore $n\ge 2$.
  Let $v_0=s,v_1,v_2,\ldots,v_m=t$ be vertices in $P$.
  Let $v_k$ be the neighbor of $v$ with maximum $k$.
  Then $G$ has a fan having at least $k+3$ vertices because $H$ is
  connected
  and $v$ has a neighbor in $H$.
  If $k\ge 6(n-1)^2-6$, then $G$ has a fan having at least $6(n-1)^2-3$
  vertices and by Lemma~\ref{lem:pathone}, $G$ contains a cycle of
  length $2n+1$ as a vertex-minor. This contradicts to our assumption
  that $G$ has no such vertex-minor. 
  Thus, $k\le 6(n-1)^2-7$.

  Let $G_0=G*v_1*v_2*v_3\cdots*v_{k-2}$
  and let $P_0=v_0v_{k-1}v_{k}v_{k+1}\cdots v_m$. 
  (If $k\le 2$, then let $G_0=G$ and $P_0=P$.)
  Then clearly $P_0$ is an induced path of $G_0$ and 
  $v_k\in N_{G_0}(v)\cap V(P_0)\subseteq \{v_0,v_{k-1},v_k\}$.

  If $N_{G_0}(v)\cap V(P_0)=\{v_{k}\}$, then we are done by taking
  $G'=G_0*v_{k-1}$ and $P'=v_0v_{k}v_{k+1}\cdots v_m$.

  If $N_{G_0}(v)\cap V(P_0)=\{v_{k-1},v_{k}\}$, then we can take
  $G'=G_0*v_{k}*v_{k-1}$ and $P'=v_0v_{k+1}v_{k+2}\cdots v_m$.

  If $N_{G_0}(v)\cap V(P_0)=\{v_0,v_{k}\}$, then we can take
  $G'=G_0*v_{k-1}*v_{k}$ and $P'=v_0v_{k+1}v_{k+2}\cdots v_m$.

  Finally,   if $N_{G_0}(v)\cap V(P_0)=\{v_0,v_{k-1},v_{k}\}$, then we can take
  $G'=G_0*v_k*v_{k-1}*v_{k+1}$ and $P'=v_0v_{k+2}v_{k+3}\cdots v_m$.

  In all cases, $\abs{E(P')}\ge \abs{E(P)}-(k+1)\ge \abs{E(P)}-6(n-1)^2+6$.
\end{proof}

\begin{LEM}\label{lem:1patch}
  Let $n\ge 2$.
  Let $G$ be a prime graph having an induced path of length $t$.
  If $t\ge 6(n-1)^2-3$, then 
  there exists a graph $G'$ locally equivalent to $G$
  having a $1$-patched induced path of length $t-6(n-1)^2+6$,
  unless $G$ has a cycle of length $2n+1$ as a vertex-minor.
\end{LEM}
\begin{proof}
  We may choose $G$ so that the length $t$ of an induced path $P$  is
  maximized among all graphs locally equivalent to $G$.
  Let $v_0,v_1,\ldots,v_m$ be vertices of $P$ in this order.
  Since $G$ is prime, $v_0$ has a neighbor $v$ other than $v_1$.
  We may assume that $v$ is non-adjacent to $v_1$ 
  because otherwise we can replace $G$ with $G*v_0$.

  Since $P$ is a longest induced path, 
  $v$ must have some neighbors in $V(P)\setminus\{v_0,v_1\}$.
  We now apply Lemma~\ref{lem:shortening} with $H=G[\{v_0,v_1\}]$, 
  deducing that there exists a graph $G'$ locally equivalent to $G$
  having a $1$-patched induced path of length $t-6(n-1)^2+6$,
  unless $G$ has a cycle of length $2n+1$ as a vertex-minor.
\end{proof}
\begin{LEM}\label{lem:patchinginductively}
  Let $n\ge 2$.
  Let $G$ be a prime graph and let $P$ be a $k$-patched induced path 
  $v_0v_1\cdots v_t$.
  If $t\ge 6(n-1)^2+k$, then 
  there exists a graph $G'$ locally equivalent to $G$
  having a $(k+1)$-patched induced path 
  $v_0v_1\cdots v_{k+2}v_iv_{i+1}\cdots v_t$
  of length at least  $t-6(n-1)^2+3$ with some $i>k+2$, unless
  $G$ has  a cycle of length $2n+1$ as a vertex-minor.
\end{LEM}
\begin{proof}
  Let $P=v_0v_1\ldots v_t$ be an induced path of length $t$ in $G$
  and $Q=w_1,w_2,\ldots,w_k$ be its $k$-patch.
  Suppose that $G$ has no vertex-minor isomorphic to a cycle of 
  length $2n+1$.

  Let $A=\{v_0,v_1,\ldots,v_{k+1}\}\cup Q$.
  By Proposition~\ref{prop:short-patch}, we may assume that 
  $G$ has a blocking sequence $b_1,b_2,\ldots,b_\ell$ of length at
  most $4$
  for $(A,V(P)\setminus A)$ because $v_{k+2}$ is the only vertex
  in $V(P)\setminus A$ having neighbors in $A$.

  Notice that $P\setminus A$ is an induced path of $G$.
  We say that a blocking sequence $b_1,b_2,\ldots,b_\ell$ for
  $(A,V(P)\setminus A)$ is \emph{nice}
  if $b_\ell$ has a unique neighbor in $V(P)\setminus A$, 
  that is also a unique neighbor of $v_{k+2}$ in $V(P)\setminus A$.

  We know that 
  $b_\ell$ has neighbors in $\{v_{k+3},\ldots,v_t\}$
  by the definition of a blocking sequence.
  We take $H=G[A\cup Q\cup \{b_1, b_2, \ldots, b_{\ell-1}\}]$.
  By Lemma~\ref{lem:shortening}, there exist a graph $G_\ell$ locally
  equivalent to $G$ and an induced path $P_\ell
  =v_0v_1\cdots v_{k+2}  v_iv_{i+1}\cdots v_t$ of $G_\ell$ 
  for some $i$ with a  $k$-patch $Q$
  such that 
  $G_\ell[A\cup\{v_{k+2}\}]=G[A\cup\{v_{k+2}\}]$,
  a sequence
  $b_1,b_2,\ldots, b_\ell$ is a nice blocking sequence for 
  $(A,V(P_\ell)\setminus A)$ in $G_\ell$,
  and 
  $\abs{E(P_\ell)}\ge t-6(n-1)^2+6$.

  Let $r\ge 1$ be minimum  such that
  there exist a graph $G'$ locally equivalent to $G$ and
  an induced path $P'=v_0 v_1\cdots v_{k+2} v_i v_{i+1}\cdots v_m$ for
  some $i$
  with a $k$-patch $Q$  in $G'$
  such that 
  $G'[A\cup\{v_{k+2}\}]=G[A\cup\{v_{k+2}\}]$,
  a sequence
  $b_1,b_2,\ldots,b_r$ is a nice blocking sequence for
  $(A,V(P')\setminus A)$ in $G'$,
  and $\abs{E(P')}\ge t-6(n-1)^2+6+r-\ell$.
  Such $r$ exists because $G_\ell$ and $P_\ell$ satisfy the condition when $r=\ell$.

  We claim that $r=1$.
  Suppose $r>1$.

  Suppose that $b_r$ is non-adjacent to $v_{k+1}$ in $G'$. Then 
  $v_i$ is the only neighbor of $b_r$ in $ V(P')$ in $G'$
  and $b_r$ is adjacent to $b_{r-1}$ in $G'$.
  If $b_{r-1}$ is non-adjacent to $v_{k+2}$, 
  then take $G''=G'*b_r$ and $P''=P'$; in $G''$, 
  a sequence $b_1,b_2,\ldots,b_{r-1}$ is a nice
  blocking sequence for $(A,V(P')\setminus A)$
  and the length of $P'$ is at least $t-6(n-1)^2+6+r-\ell$.
  This leads a contradiction to the assumption that $r$ is minimized. 
  Therefore $b_{r-1}$ is adjacent to $v_{k+2}$.
  Then take $G''=G'*b_r*v_i$ with $P''=v_0v_1\cdots v_{k+2} v_{i+1} \cdots
  v_m$.
  Then $b_1,b_2,\ldots,b_{r-1}$ is a nice blocking sequence for
  $(A,V(P'')\setminus A)$ in $G''$ and the length of $P''$ is at least 
  $t-6(n-1)^2+6+r-\ell-1$.  This contradicts to the assumption that
  $r$ is chosen to be minimum.

  Therefore $b_r$ is adjacent to $v_{k+1}$ in $G'$. Since
  $b_r$ is the last vertex in the blocking sequence, 
  $b_r$ is also adjacent to $w_k$ in $G'$.
  If $b_{r-1}$ is non-adjacent to $v_{k+2}$, 
  then take $G''=G'*v_{k+2}*b_r$ and $P''=P'$; in $G''$, 
  a sequence $b_1,b_2,\ldots,b_{r-1}$ is a nice
  blocking sequence for $(A,V(P'')\setminus A)$
  and the length of $P''$ is at least $t-6(n-1)^2+6+r-\ell$,
  contradicting our assumption on $r$.
  So $b_{r-1}$ is adjacent to $v_{k+2}$.
  Then we take $G''=G'*v_{k+2}*b_r*v_i$ with $P''=v_0v_1\cdots v_{k+2} v_{i+1} \cdots
  v_m$.
  Then $b_1,b_2,\ldots,b_{r-1}$ is a nice blocking sequence for
  $(A,V(P'')\setminus A)$ in $G''$ and the length of $P''$ is at least 
  $t-6(n-1)^2+6+r-\ell-1$.  This again  contradicts to the assumption
  on $r$.
  This proves that $r=1$.

  Since $b_1$ is a nice blocking sequence for $(A,V(P')\setminus A)$
  in $G'$,
  $b_1$ has a neighbor in $A$ in $G'$
  and $N_{G'}(b_1)\cap A\neq \{v_{k+1},w_k\}$.
  In addition, $v_i$ is the only neighbor of $b_1$ among $V(P')\setminus A$ in $G'$.
  Now it is easy to see that $w_1,w_2,w_3,\ldots,w_k,b_1$ is a
  $(k+1)$-patch of $P'$ in $G'$.
  And, since $\ell\le 4$, we have $\abs{E(P')}\ge t-6(n-1)^2+3$.
\end{proof}
\begin{PROP}\label{prop:patchingpath}
  Let $N\ge 4$ be an integer.
  If a prime graph $G$ on at least $5$ vertices
  has an induced path of length 
  $L=(6(n-1)^2-2)(N-2)-1$,
  then there exists a graph $G'$ locally equivalent to $G$ having a
  fully patched
  induced path of length $N$,
  unless $G$ has a cycle of length $2n+1$ as a vertex-minor.
\end{PROP}
\begin{proof}
  Suppose that $G$ has no cycle of length $2n+1$ as a vertex-minor.
  Then $n\ge 3$ by Theorem~\ref{thm:bouchet}.
  By Lemma~\ref{lem:1patch}, we may assume that $G$
  has a $1$-patched path of length $L-6(n-1)^2+6$.
  By Lemma~\ref{lem:patchinginductively}, 
  we may assume that $G$ has an $(N-2)$-patched path of length 
  \[
  L-6(n-1)^2+6- (N-3)(6(n-1)^2-3)=N
  \]
  Thus $G$ has a fully patched induced path of length $N$.
\end{proof}
\subsection{Finding a cycle from a fully patched path}
We aim to find a cycle as a vertex-minor
in a sufficiently long fully patched path.

Let $P=v_0v_1\cdots v_n$ be an induced path of a graph
$G$
with a $(n-2)$-patch $Q=w_1w_2w_3,\ldots w_{n-2}$.
Let $A_1=\{v_0,v_1\}$ and 
for $i=2,\ldots,n-2$, let $A_i=\{v_0,v_1,\ldots,v_i,w_1,w_2,\ldots,w_{i-1}\}$
and $B_i=V(P)\setminus A_i$ for all $i\in\{1,2,\ldots,n-2\}$.

For $i\ge 1$, let $L(w_i)$ be the minimum $j\ge 0$ such that 
\[ \rho_G^*(A_{j+1},B_{j+1}\cup \{w_i\})>1.\]
Since $w_i$ is a blocking sequence for $(A_i,B_i)$, 
$L(w_i)$ is well defined and $L(w_i)<i$.

We classify vertices in $Q$ as follows.
\begin{itemize}
\item A vertex $w_i$ has \emph{Type 0} if $L(w_i)=0$ and $w_i$ is adjacent to $v_0$.
\item A vertex $w_i$ has \emph{Type 1} if $L(w_i)\ge 1$ and $w_i$
  has  no neighbor in $A_{L(w_i)}$
  and $w_i$ is adjacent to exactly one of $v_{L(w_i)+1}$ and $w_{L(w_i)}$.
\item A vertex $w_i$ has \emph{Type 2} if $L(w_i)=1$ and $w_i$ is
  adjacent to $v_1$, non-adjacent to $v_0$.
\item A vertex $w_i$ has \emph{Type 3} if $L(w_i)\ge 2$ and $w_i$ 
  has no neighbor in $A_{L(w_i)-1}$ 
  and $w_i$ is adjacent to both $v_{L(w_i)}$ and $w_{L(w_i)-1}$.
\end{itemize}
By the definition of fully patched paths, we can deduce the following
lemma easily.
\begin{LEM}
  Each vertex in $Q$ has Type 0, 1, 2, or 3.
\end{LEM}
\begin{proof}
  If $w_i$ is adjacent to $v_0$, then $\rho_G^*(A_1,B_1\cup\{w_i\})>1$ and therefore
  $L(w_i)=0$, implying that $w_i$ has Type 0.
  We may now assume that $w_i$ is non-adjacent to $v_0$ and so
  $L(w_i)>0$.
  
  If $w_i$ has no neighbors in $A_{L(w_i)}$, then 
  $\rho_G^*(A_{L(w_i)+1},B_{L(w_i)+1}\cup\{w_i\})
  =\rho_G^*(A_{L(w_i)+1}\setminus
  A_{L(w_i)},B_{L(w_i)+1}\cup\{w_i\})>1$.
  Thus $v_{L(w_i)+2}$ and $w_i$ cannot have the same set of
  neighbors in  $A_{L(w_i)+1}\setminus
  A_{L(w_i)}=\{v_{L(w_i)+1},w_{L(w_i)}\}$.
  By the definition of fully patched paths, 
  $v_{L(w_i)+2}$ is adjacent to both $v_{L(w_i)+1}$ and $w_{L(w_i)}$.
  It follows that $w_i$ is adjacent to exactly one of $v_{L(w_i)+1}$
  and $w_{L(w_i)}$.
  So $w_i$ has Type 1.

  Now we may assume that $w_i$ has some neighbors in $A_{L(w_i)}$.
  By definition, 
  \[\rho_G^*(A_{L(w_i)},B_{L(w_i)}\cup\{w_i\})\le 1\]
  and therefore $w_i$ and $v_{L(w_i)+1}$ have the same set of
  neighbors in $A_{L(w_i)}$. 
  Therefore, if $L(w_i)=1$, then $w_i$ is adjacent to $v_1$, implying
  that $w_i$ has Type 2.
  If $L(w_i)>1$, then $w_i$ is adjacent to both $v_{L(w_i)}$ and
  $w_{L(w_i)-1}$, and so $w_i$ has Type 3.
\end{proof}

We say that a pair of paths $P^i_1$ and $P^i_2$ from
$\{v_0,v_1\}$ to $\{v_{i+1},w_i\}$ 
is \emph{good}
if
\begin{enumerate}[(i)]
\item 
$P^i_1$ and $P^i_2$ are vertex-disjoint induced paths on $ A_{i+1}$,
\item 
 for each $j\in \{1,2,\ldots,i-1\}$,
$w_j\in V(P^i_1)\cup V(P^i_2)$ or $v_{j+1}\in V(P^i_1)\cup V(P^i_2)$,
\item $G[ V(P^i_1)\cup V(P^i_2)]+v_{i+1}w_i$ is a generalized ladder with
  two defining paths $P^i_1$ and $P^i_2$.
\end{enumerate}

\begin{LEM}\label{lem:findladder}
  For all $i\in \{1,2,\ldots,n-2\}$, 
  $G$ has a good pair of paths $P^i_1$ and $P^i_2$
  from $\{v_0, v_1\}$ to $\{v_{i+1}, w_i\}$.
\end{LEM}
\begin{proof}
We proceed by induction on $i$.
  If $w_i$ has Type 0, then
  let $P^i_1=v_1v_2 \cdots v_{i+1}$ and $P^i_2=v_0w_i$.
  Since $v_0$ has no neighbors in $\{v_2, v_3, \ldots, v_{i+1}\}$, 
  $G[ V(P^i_1)\cup V(P^i_2)]+v_{i+1}w_i$ is a generalized ladder with
  two defining paths $P^i_1$ and $P^i_2$.
  Also, $V(P^i_1)\cup V(P^i_2)\subseteq A_{i+1}$ and for all $j\in \{1, 2, \ldots, i-1\}$, $v_{j+1}\in V(P^i_1)$. 
  Thus, the pair $(P^i_1,P^i_2)$ is good.
  
  If $w_i$ has Type 2, then 
  let $P^i_1=v_0w_1v_3v_4 \cdots v_{i+1}$ and $P^i_2=v_1w_i$.
  By the definition of a patched path,
  $v_1$ is not adjacent to $w_1$.
  So, $v_1$ has no neighbors in $\{w_1, v_3, v_4, \ldots, v_{i+1}\}$, and
   therefore $G[ V(P^i_1)\cup V(P^i_2)]+v_{i+1}w_i$ is a generalized ladder with
  two defining paths $P^i_1$ and $P^i_2$.
  Clearly, $V(P^i_1)\cup V(P^i_2)\subseteq A_{i+1}$.
  Moreover, $w_1\in V(P^i_1)$ and for each $j\in \{2,\ldots,i-1\}$,
	$v_{j+1}\in V(P^i_1)$.
	Therefore, the pair $(P^i_1,P^i_2)$ is good.

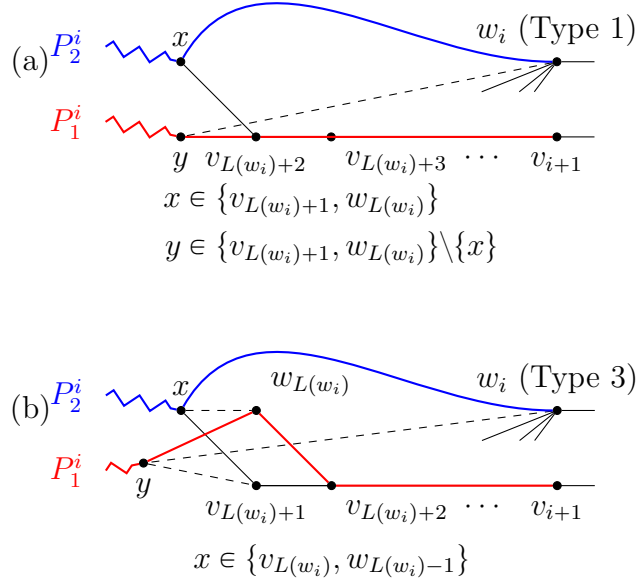
\begin{figure}
  \centering
  \tikzstyle{v}=[circle, draw, solid, fill=black, inner sep=0pt, minimum width=3pt]
   \begin{tikzpicture}
    \draw (1,0)--(6.5,0);
    \node [v,label=$w_i$ (Type 1)] (wi) at (6,1){};
    \node [v,label=below:$v_{i+1}$] (vi1) at (6,0){};
    \draw (wi)--(6.5,1);
    \node [v,label=$x$] (wj) at (1,1){};
    \node [v,label=below:$y$] (vj1) at (1,0){};
    \node [v,label=below:$v_{L(w_i)+2}$] (vj2) at (2,0){};
    \node [v,label=below right:$v_{L(w_i)+3}\,\,\cdots$] (vj3) at (3,0){};
    \draw (wj)--(vj2);
    \draw[dashed] (wi)--(vj1);
    \draw (wi)--(5,0.6);
    \draw (wi)--(5.5,0.6);
    \draw (wi)--(5.7,0.6);
    \draw node at  (-1,1) {(a)};
    \draw (0,1.2) [snake=zigzag,thick,blue] node [label=left:$P^i_2$]{} --
    (wj); \draw  [blue,thick, out=60,in=180] (wj) to  (wi);
    \draw [snake=zigzag,thick,red] (0,.2) node[label=left:$P^i_1$] {}--
    (vj1);\draw [red,thick] (vj1)-- (vj2)--(vi1);
      \node  (wi) at (3,-1){        \begin{minipage}{5cm}
          \begin{align*}
            x&\in \{v_{L(w_i)+1}, w_{L(w_i)}\}\\
            y&\in \{v_{L(w_i)+1}, w_{L(w_i)}\}\setminus \{x\}
          \end{align*}
        \end{minipage}};
        \end{tikzpicture}
  \begin{tikzpicture}
    \draw (2,0)--(6.5,0);
    \node [v,label=$w_i$ (Type 3)] (wi) at (6,1){};
    \node [v,label=below:$v_{i+1}$] (vi1) at (6,0){};
    \draw (wi)--(6.5,1);
    \node [v,label=$x$] (wj) at (1,1){};
    \node [v,label=below:$y$] (vj1) at (0.5,0.3){};
    \node [v,label=below:$v_{L(w_i)+1}$] (s) at (2,0){};
    \node [v,label=above right:$w_{L(w_i)}$] (vj2) at (2,1){};
    \node [v,label=below right:$v_{L(w_i)+2}\,\,\cdots$] (vj3) at (3,0){};
    \draw (wj) to (s) to (vj3);
    \draw [dashed] (wj)--(vj2);
    \draw[dashed] (wi)--(vj1);
    \draw[dashed] (vj1)--(s);
    \draw (wi)--(5,0.6);
    \draw (wi)--(5.5,0.6);
    \draw (wi)--(5.7,0.6);
    \draw node at  (-1,1) {(b)};
    \draw (0,1.2) [snake=zigzag,thick,blue] node [label=left:$P^i_2$]{} --
    (wj); \draw  [blue,thick,out=60,in=180] (wj) to (wi);
    \draw [snake=zigzag,thick,red] (0,.2) node[label=left:$P^i_1$] {}--
    (vj1);\draw [red,thick] (vj1)-- (vj2) to (vj3) to (vi1);
    \node  (wi) at (3,-1){$x\in \{v_{L(w_i)}, w_{L(w_i)-1}\}$};
  \end{tikzpicture}

 \caption{Constructing a generalized ladder in a fully patched path. The vertex $w_i$ has Type 1 in (a) and has Type 3 in (b).}
  \label{fig:findladder}
\end{figure}

   Now, we may assume that $w_i$ has Type 1 or Type 3.
   Since $L(w_i)\ge 1$, by the induction hypothesis, 
   $G$ has a good pair of paths $P^{L(w_i)}_1$, $P^{L(w_i)}_2$ from $\{v_0, v_1\}$ to $\{v_{L(w_i)+1}, w_{L(w_i)}\}$.
     
   Suppose $w_i$ has Type 1 and therefore $w_i$ is adjacent to exactly one of $v_{L(w_i)+1}$ and $w_{L(w_i)}$.
	Let $\{x,y\}=\{v_{L(w_i)+1}, w_{L(w_i)}\}$ such that $x$ is adjacent to $w_i$.
	We may assume that the paths $P^{L(w_i)}_1$ and $P^{L(w_i)}_2$ end at 
	$y$ and $x$, respectively.	
	Let $P^i_1$ be a path 
	\[P^{L(w_i)}_1+yv_{L(w_i)+2}v_{L(w_i)+3} \cdots v_{i+1}\]
	and let $P^i_2$ be a path $P^{L(w_i)}_2+xw_i$.
	See Figure~\ref{fig:findladder}.
	By the induction hypothesis, 
	$V(P^{L(w_i)}_1)\cup V(P^{L(w_i)}_2)\subseteq A_{L(w_i)+1}\subseteq A_{i+1}$, and  
	for each $j\in \{1,2,\ldots,L(w_i)-1\}$,
        $V(P^{L(w_i)}_1)\cup V(P^{L(w_i)}_2)$ contains $w_j$ or $v_{j+1}$.
	Thus it follows that $V(P^i_1)\cup V(P^i_2)\subseteq A_{i+1}$ and 
	for each $j\in \{1,2,\ldots,i-1\}$,
        $V(P^{i}_1)\cup V(P^{i}_2)$ contains $w_j$ or $v_{j+1}$.
	
	We claim that $G[ V(P^i_1)\cup V(P^i_2)]+v_{i+1}w_i$ is a generalized ladder with the defining paths $P^i_1$ and $P^i_2$.
	By the induction hypothesis, it is enough to show that
	there are no two crossing chords $xa$ and $w_ib$ for some $a,
        b\in V(P^i_1)$.
        Since 
        $w_i$ has no neighbor in $A_{L(w_i)}$
        and $w_i$ and $y$ are non-adjacent,
        $b\in X=\{v_{k}: k\in \{L(w_i)+2,L(w_i)+3,\ldots,i+1\}\}$.
        Since $x$ has no neighbor in $X\setminus\{v_{L(w_i)+2}\}$, we
        deduce that $xa$ and $w_ib$ cannot cross and therefore 
        $G[ V(P^i_1)\cup V(P^i_2)]+v_{i+1}w_i$ is a generalized ladder.
	This proves that if $w_i$ has Type 1, then $(P^i_1,P^i_2)$ is
        a good pair.

  Finally, suppose that $w_i$ has Type 3
  and so  $w_i$ is adjacent to both $v_{L(w_i)}$ and
  $w_{L(w_i)-1}$.
  By symmetry, we may assume that $P^{L(w_i)}_2$ ends at
  $v_{L(w_i)+1}$.
  Let $x$ be the predecessor of $v_{L(w_i)+1}$ in $P^{L(w_i)}_2$.
  Since $P^{L(w_i)}_2$ is on $A_{L(w_i)+1}$ and $v_{L(w_i)+1}$ has
  only two neighbors $v_{L(w_i)}$, $w_{L(w_i)-1}$ in $A_{L(w_i)+1}$, 
  either $x=v_{L(w_i)}$ or
  $x=w_{L(w_i)-1}$.
  Let $y$ be the predecessor of $w_{L(w_i)}$ in $P^{L(w_i)}_1$.
  	Let $P^{i}_1$ be a path \[P^{L(w_i)}_1+w_{L(w_i)}v_{L(w_i)+2}v_{L(w_i)+3} \cdots v_{i+1}\] 
	and let $P^{i}_2$ be a path obtained from $P^{L(w_i)}_2$ by removing $v_{L(w_i)+1}$ and adding $xw_i$. See Figure~\ref{fig:findladder}(b).
	It follows from our construction and the induction hypothesis
        that 
	$V(P^i_1)\cup V(P^i_2)\subseteq A_{i+1}$ and 
        $V(P^{i}_1)\cup V(P^{i}_2)$ contains $w_j$ or $v_{j+1}$
	for each $j\in \{1,2,\ldots,i-1\}$.
   
   We claim that $G[ V(P^i_1)\cup V(P^i_2)]+v_{i+1}w_i$ is a generalized ladder with the defining paths $P^i_1$ and $P^i_2$.
   By the induction hypothesis, it is enough to prove that 
   there are no two chords $xa$ and $w_ib$ such that $a, b\in V(P^i_1)$
   and $b$ precedes $a$ in $P^i_1$.   
   Suppose not.
   Since $w_i$ has no neighbor in $A_{L(w_i)-1}$,
   neighbors of $w_i$ in $P^i_1$ are in $\{y,w_{L(w_i)}\}\cup\{v_{k}:
   k\in \{L(w_i)+2,L(w_i)+3,\ldots,i+1\}\}$.
   Since $x$ has no neighbor in $\{v_{k}:
   k\in \{L(w_i)+2,L(w_i)+3,\ldots,i+1\}\}$, we deduce that 
   $a=w_{L(w_i)}$ and $b=y$. Since $w_i$ has no neighbor in
   $A_{L(w_i)-1}$, $b$ is one of $v_{L(w_i)}$ and $w_{L(w_i)-1}$ other
   than $x$.
   Thus $w_{L(w_i)}$ is adjacent to both $v_{L(w_i)}$ and
   $w_{L(w_i)-1}$.	
   This contradicts (iii) because $v_{L(w_i)+1}$ is also 
   adjacent to both $v_{L(w_i)}$ and
   $w_{L(w_i)-1}$ and so $G[V(P^{L(w_i)}_1)\cup
   V(P^{L(w_i)}_2)]+v_{L(w_i)+1}w_{L(w_i)}$ is not a generalized ladder.
\end{proof}

\begin{LEM}\label{lem:path-to-ladder}
  If a graph has  a fully patched induced path of length $n$,
  then it has a generalized ladder having at least $n+2$ vertices
  as an induced subgraph.
\end{LEM}
\begin{proof}
  Let $P=v_0v_1\cdots v_n$ be the induced path of length $n$
  with  an $(n-2)$-patch $Q=w_1w_2\cdots w_{n-2}$.
  Lemma~\ref{lem:findladder} provides a good pair of paths $P^{n-2}_1$ and $P^{n-2}_2$
  from $\{v_0, v_1\}$ to $\{v_{n-1}, w_{n-2}\}$ 
  such that $G[V(P^{n-2}_1)\cup V(P^{n-2}_2)]+v_{n-1}w_{n-2}$ is a generalized ladder
  and 
  $V(P^{n-2}_1)\cup V(P^{n-2}_2)$ contains 
  $w_j$  or $v_{j+1}$
  for each $j\in \{1,2,\ldots,n-3\}$.
  Since $v_n$ is only adjacent to $v_{n-1}$ and $w_{n-2}$ in $G$,
  $G'=G[V(P^{n-2}_1)\cup V(P^{n-2}_2)\cup \{v_n\}]$ is a generalized ladder.
  Since $v_0,v_1,v_n, v_{n-1}, w_{n-2}\in V(G')$, 
  $G'$ has at least $(n-3)+5=n+2$ vertices.
\end{proof}

Now we are ready to prove the main theorem of this section.
\begin{LEM}\label{lem:path-to-cycle}
  Let $n\ge 1$.
  If a prime graph has an induced path of length 
  $110592 n^7$, then 
  it has a cycle of length $4n+3$ as a vertex-minor.
\end{LEM}

\begin{proof}%
  Let $G$ be a prime graph having an induced path of length $110592n^7$.
  Suppose that $G$ has no cycle of length $4n+3$ as a vertex-minor.
  Let $N=4608n^5$.
  Then 
  \[
  (6(2n)^2-2)(N-2)-1 <  110592n^7.
  \]
  Thus by Proposition~\ref{prop:patchingpath},
  there exists a graph $G'$ locally equivalent to $G$
  having a fully patched induced path of length $N$.
  By Lemma~\ref{lem:path-to-ladder}, $G'$ must have a generalized
  ladder having at least $N+2$ vertices as an induced subgraph. 
  By Proposition~\ref{prop:ladder}, we deduce that $G'$ has a cycle of
  length $4n+3$ as a vertex-minor.
\end{proof}
\begin{proof}[Proof of Theorem~\ref{thm:path-to-cycle}]
  Let $k=\lfloor n/4\rfloor$.
  Let $G$ be a prime graph having a path of length at least $6.75n^7$.
  Then  $G$ has a path of length $6.75(4k)^7= 110592k^7$,
  and by Lemma~\ref{lem:path-to-cycle}, $G$ has a cycle of length
  $4k+3\ge n$ as a vertex-minor.
\end{proof}

\section{Main Theorem}\label{sec:main}

In this section, we prove the following.
\begin{THM}\label{thm:mainthm}
  For every $n$, there is $N$ such that
  every prime graph on at least $N$ vertices has a vertex-minor
  isomorphic to $C_n$ or $\K_n\mat\K_n$. 
\end{THM} 
By Theorem~\ref{thm:path-to-cycle}, it is enough to prove the
following proposition.
\begin{PROP}\label{prop:secexclude}
  For every $c$, there exists $N$ such that 
  every prime graph on at least $N$ vertices 
  has a vertex-minor isomorphic to either $P_c$ or $\K_{c}\mat\K_{c}$.
\end{PROP}

Here is the proof of Theorem~\ref{thm:mainthm} assuming Proposition \ref{prop:secexclude}.
\begin{proof}[Proof of Theorem~\ref{thm:mainthm}]
  We take $c=\lceil 6.75n^7\rceil$ and apply
  Proposition~\ref{prop:secexclude} and Theorem \ref{thm:path-to-cycle}.
\end{proof}

For integers $h,w,\ell\ge 1$, 
a \emph{$(h,w,\ell)$-broom} of a graph $G$
is a connected induced subgraph $H$ of $G$
such that 
\begin{enumerate}[(i)]
\item $H$ has an induced path $P$ of length $h$ from some vertex $v$
  called the \emph{center}, 
\item $P\setminus v$ is a component of $H\setminus v$, 
\item $H\setminus V(P)$ has $w$ components, each having exactly $\ell$ vertices.
\end{enumerate}
The path $P$ is called a \emph{handle} of $H$ and 
each component of $H\setminus V(P)$ is called a \emph{fiber} of $H$.
If $H=G$, then we say that $G$ is a $(h,w,\ell)$-broom.
We call $h$, $w$, $\ell$  the \emph{height}, \emph{width}, \emph{length},
respectively,  of a~$(h,w,\ell)$-broom. See Figure~\ref{fig:broom}.
Observe that $v$ has one or more neighbors in each fiber.

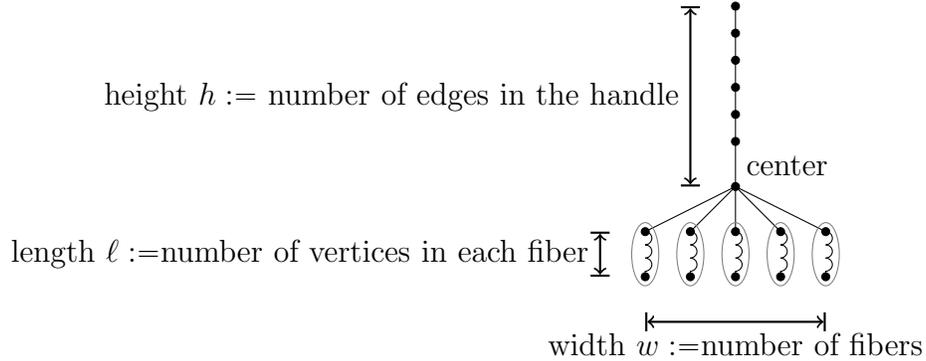
\begin{figure}
  \tikzstyle{v}=[circle, draw, solid, fill=black, inner sep=0pt, minimum width=3pt]
  \begin{tikzpicture}[scale=0.06,rotate=-90]
    \draw (0,30) node [above right] {center} {};
    \draw (0,30) node [v] {} 
    \foreach \x in {-10,-16,..., -40}
    {   -- (\x,30) node[v]{} };
    \draw[|<->|,thick] (-40,20)
    -- node[left,sloped] {height $h:=$ number of edges in the handle} (0,20);
    \draw[|<->|,thick] (30,50)
    -- node[below] {width $w:=$number of fibers} (30,10);
    \draw[|<->|,thick] (10,0)
    -- node[left,sloped] {length $\ell:=$number of vertices in each fiber} (20,0);
    
    \foreach \x in {10,20,..., 50}
    {
      \draw (0,30) -- (10,\x)node[v]{} [snake=bumps]--(20,\x) node [v]{};
      \draw [color=gray]  (15,\x) ellipse (7 and 3);
    }
  \end{tikzpicture}
  \caption{A~$(h,w,\ell)$-broom. }\label{fig:broom}
\end{figure}

Here is the rough sketch of the proof. 
If a prime graph $G$ 
has no vertex-minor isomorphic to $P_c$ or $\K_c\mat\K_c$ and
$G$ has a broom having huge width as a vertex-minor, 
then it has a vertex-minor isomorphic to a broom with larger length
and sufficiently big width.
So, we increase the length of a broom while keeping its width
big.
If we obtain a broom of big length by repeatedly applying this process,
then we will obtain a broom of larger height.
By growing the height, we will eventually obtain a long
induced path.

To start the process, we need an initial broom with sufficiently big width.
For that purpose, we use the following Ramsey-type theorem.
\begin{THM}[%
  folklore; see Diestel~{\cite[Proposition 9.4.1]{Diestel2010}}]\label{thm:binary}
  For positive integers $c$ and $t$, there exists $N=g_0(c,t)$ such that 
  every connected graph on at least $N$ vertices must contain 
  $\K_{t+1}$, $\K_{1,t}$, or $P_c$ as an induced subgraph.
\end{THM}

By Theorem~\ref{thm:binary}, 
if $G$ is prime and $\abs{V(G)}\ge g_0(c,t+1)$, then either $G$ has an induced subgraph isomorphic to $P_c$ or 
	$G$ has a vertex-minor isomorphic to $K_{1,t+1}$. 
	Since a~$(1,t,1)$-broom is isomorphic to $K_{1,t+1}$,
	we conclude that every sufficiently large prime graph has a vertex-minor isomorphic to a
        $(1,t,1)$-broom, unless it has an induced subgraph isomorphic
        to $P_c$.
\subsection{Increasing the length of a broom}
We now show that 
if a prime graph has a broom having sufficiently large width, 
we can find a broom having larger length
after applying local complementation and shrinking the width.

In the following proposition, we want to find a wide broom of length $2$ when we are
given a sufficiently wide broom of length $1$,
when the graph has no 
$P_c$
or $\K_c\mat\K_c$ as  a vertex-minor.
\begin{PROP}\label{prop:lengthone}
  For all integers $c\ge 3$ and $t\ge 1$, 
  there exists $N=g_1(c,t)$ such that 
  for each $h\ge 1$, 
  every prime graph having a~$(h,N,1)$-broom %
  has a vertex-minor isomorphic to a~$(h,t,2)$-broom, 
  $\K_c\mat\K_c$, or $P_c$.
\end{PROP}
We will use the following theorem.
\begin{THM}[Ding, Oporowski, Oxley,
  Vertigan~\cite{DOOV1996}]\label{thm:largebipartite}
     For every positive integer $n$,
     there exists $N=f(n)$ 
     such that for every bipartite graph $G$ with a bipartition $(S,T)$, 
     if no two vertices in $S$ have the same set of neighbors
     and      $\abs{S}\ge N$,  
     then $S$ and $T$ have $n$-element subsets $S'$ and $T'$, respectively, such that 
     $G[S',T']$ is isomorphic to $\S_{n}\mat\S_{n}$, $\S_{n}\tri\S_{n}$, or
     $\S_{n}\antimat\S_{n}$. 
\end{THM}

\begin{proof}[Proof of Proposition~\ref{prop:lengthone}]

	Let $N=f(R(w,w))$ where 
	$f$ is the function in Theorem~\ref{thm:largebipartite},
	and $w=\max(t+(c-1)(c-3),2c-1)$.
	Suppose that $G$ has a~$(h,g_1(c,t),1)$-broom $H$. 
	Note that every fiber of $H$ is a
        single vertex. 

	Let $S$ be the union of the vertex sets of all fibers of $H$, and  $x$ be the center of $H$. 
	Let $N_G(S)\setminus \{x\}=T$. 
	Since $G$ is prime, no two vertices in $G$ have the same set of neighbors, 
	and so two distinct vertices in $S$ have different sets of neighbors in $T$.
	Since $\abs{S}=N=f(R(w,w))$,
	by Theorem~\ref{thm:largebipartite}, 
	there exist $S_0\subseteq S$, $T_0\subseteq T$ such that
	$G[S_0,T_0]$ is isomorphic to
	$\S_{R(w,w)}\mat\S_{R(w,w)}$, $\S_{R(w,w)}\tri\S_{R(w,w)}$ or $\S_{R(w,w)}\antimat\S_{R(w,w)}$. 
	Since $\abs{T_0}\ge R(w,w)$, 
	by Ramsey's theorem, there exist $S'\subseteq S_0$ and
        $T'\subseteq T_0$ such that 
	$G[S',T']$ is isomorphic to 
	$\S_{w}\mat\S_{w}$, $\S_{w}\tri\S_{w}$, or $\S_{w}\antimat\S_{w}$,
	 and $T'$ is a clique or a stable set in $G$.
	If $G[S',T']$ is isomorphic to $\S_{w}\tri\S_{w}$ or $\S_{w}\antimat\S_{w}$, 
   then by Lemmas~\ref{lem:lengthonecase1} and \ref{lem:lengthonecase2}, $G$ has a vertex-minor isomorphic to either $P_{2w}$ or $\K_{w-2}\mat\K_{w-2}$.
   Since $w\ge 2c-1$ and $c\ge 3$, we have $P_c$ or $\K_c\mat\K_c$.
	Thus we may assume that $G[S',T']$ is isomorphic to
        $\S_{w}\mat\S_{w}$.

    If $T'$ is a clique in $G$, then 
    we can remove the edges connecting $T'$ with $x$ by applying local complementation at some vertices in $S'$. 
    Thus, we can obtain a vertex-minor isomorphic to $\K_{w}\mat\K_w$ from $G[S'\cup T'\cup \{x\}]$ by applying local complementation at $x$ and deleting $x$.
    Therefore we may assume that $T'$ is a stable set in $G$.  

    We claim that each vertex $y\neq x$ in the handle of $H$ is adjacent to
    at most $c$ vertices in $T'$, or $G$ has $\K_c\mat\K_c$ as a vertex-minor.
    Suppose not. 
    If $y$ is a neighbor of $x$, then by pivoting an edge of
    $G[S',T']$, we can delete the edge $xy$. From there, we
    obtain a vertex-minor isomorphic to $\K_c\mat\K_c$
    by applying local complementation at $x$ and $y$.
    If $y$ is not adjacent to $x$, 
    then we obtain a vertex-minor
    isomorphic to $\K_c\mat\K_c$ 
    by deleting all vertices in the handle other than $x$ and $y$,
    and applying local complementation at $x$ and $y$.
    This proves the claim.

    By deleting at most $(c-1)h$ vertices in $T'$ and their pairs
    in $S'$, we can assume that no vertex other than $x$ in the handle has a neighbor
    in $T'$ and this broom has width at least $w-(c-1)h$.
    If $h+2\ge c$, then we have $P_c$ as an induced subgraph. 
    Thus we may assume that $h\le c-3$.
    Since $w-(c-1)h\ge w-(c-1)(c-3)\ge t$, we obtain a vertex-minor isomorphic to a~$(h,t,2)$-broom. 
\end{proof}

We now aim to increase the length of a broom when the broom has length
at least $2$.
For a fiber $F$ of a broom $H$,
we say that a vertex $v\in V(G)\setminus V(H)$ \emph{blocks} $F$ 
if \[\rho_G^*(V(F),(V(H)\setminus V(F))\cup\{v\})>1.\]
 If $G$ is prime
and $F$ has at least two vertices, then $G$ has a blocking sequence for
$(V(F),V(H)\setminus V(F))$ by Proposition~\ref{prop:block}
and therefore there exists a vertex $v$ that blocks $F$
because we can take the first vertex in the blocking sequence.

\begin{LEM}\label{lem:blockone}
  Let $G$ be a  graph 
  and let $x,y$ be  two vertices such that 
  $\rho_G(\{x,y\})=2$
  and $G\setminus x\setminus y$ is connected.
  Then there exists some sequence $v_1,v_2,\ldots,v_n\in
  V(G)\setminus\{x,y\}$ of (not necessarily distinct) vertices such
  that
  $G*v_1*v_2\cdots *v_n $ has an induced path of length $3$ from $x$
  to $y$.
\end{LEM}
\begin{proof}
  We proceed by induction on $\abs{V(G)}$.
  If $\abs{V(G)}=4$, then it is easy to check all cases to obtain a
  path of length $3$. To do so, first observe that up to symmetry, there are
  $2$ cases in $G[\{x,y\}, V(G)\setminus\{x,y\}]$; either it is a
  matching of size $2$ or a path of length $3$. In both cases, one can find a
  desired sequence of vertices to apply local complementation, see
  Figure~\ref{fig:4vertex} for all possible graphs on $4$-vertices up
  to isomorphism.
  \begin{figure}
    \centering
    \tikzstyle{v}=[circle, draw, solid, fill=black, inner sep=0pt, minimum width=3pt]
    \begin{tikzpicture}
      \begin{scope}[xshift=0cm]
      \node[v] (x) at (0,0) {};
      \node[v] (y) at (1,0) {};
      \node[v] (z) at (1,1) {};
      \node[v] (w) at (0,1) {};
      \draw[dashed](-1,.5)--(1.5,.5);
      \draw (x)--(w)--(z)--(y);
      \node [left ] at (0,0)  {$x$, $y$};
      \end{scope}
      \begin{scope}[xshift=3cm]
      \node[v] (x) at (0,0) {};
      \node[v] (y) at (1,0) {};
      \node[v] (z) at (1,1) {};
      \node[v] (w) at (0,1) {};
      \draw[dashed](-1,.5)--(1.5,.5);
      \node [left ] at (0,0)  {$x$, $y$};
      \draw (x)--(w)--(z)--(y)--(x);
      \end{scope}
      \begin{scope}[xshift=6cm]
      \node[v] (x) at (0,0) {};
      \node[v] (y) at (1,0) {};
      \node[v] (z) at (1,1) {};
      \node[v] (w) at (0,1) {};
      \draw[dashed](-1,.5)--(1.5,.5);
      \node [left ] at (0,0)  {$x$, $y$};
      \draw (x)--(w)--(z)--(y);
      \draw (w)--(y);
      \end{scope}
      \begin{scope}[xshift=9cm]
      \node[v] (x) at (0,0) {};
      \node[v] (y) at (1,0) {};
      \node[v] (z) at (1,1) {};
      \node[v] (w) at (0,1) {};
      \draw[dashed](-1,.5)--(1.5,.5);
      \node [left ] at (0,0)  {$x$, $y$};
      \draw (x)--(w)--(z)--(y)--(x);
      \draw (w)--(y);
      \end{scope}
    \end{tikzpicture}
    \caption{Dealing with $4$-vertex graphs in Lemma~\ref{lem:blockone}.}
    \label{fig:4vertex}
  \end{figure}
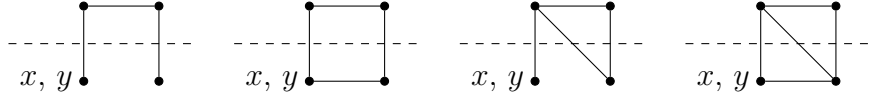

  Now we may assume that $G$ has at least $5$ vertices.
  Let $A_1=N_G(x)\setminus (N_G(y)\cup \{y\})$, $A_2=N_G(x)\cap N_G(y)$, 
  and $A_3=N_G(y)\setminus (N_G(x)\cup\{x\})$. 
  Clearly $\rho_G(\{x,y\})=2$ is equivalent to say that at least two
  of $A_1$, $A_2$, $A_3$ are nonempty.

  We say a vertex $t$ in $G\setminus x\setminus y$ \emph{deletable} if
  $G\setminus x\setminus y\setminus t$ is connected.
  If there is a deletable vertex not in $A_1\cup A_2\cup A_3$, then 
  $\rho_{G\setminus t}(\{x,y\})=2$ and we apply the induction
  hypothesis to find an induced path.
  Thus we may assume that all deletable vertices are in $A_1\cup
  A_2\cup A_3$.

  If $\abs{A_i}>1$ and $A_i$ has  a deletable vertex $t$ for some
  $i=1,2,3$, then 
  $\rho_{G\setminus t}(\{x,y\})=2$ and so we obtain a sequence by
  applying the induction hypothesis. 
  So we may assume that 
  if $A_i$ has a deletable vertex, then $\abs{A_i}=1$.

  If there are three deletable vertices $t_1$, $t_2$, $t_3$
  in $G\setminus x\setminus y$,   then we may assume $A_i=\{t_i\}$. 
  However, $\rho_{G\setminus t_1}(\{x,y\})=2$ because $A_2$, $A_3$ are
  nonempty and therefore we obtain an induced path from $x$ to $y$ by
  the induction hypothesis.

  Thus we may assume that $G\setminus x\setminus y$ has at most $2$
  deletable vertices.
  So  $G\setminus x\setminus y$ has maximum degree
  at most $2$ 
  because otherwise 
  we can choose leaves of a spanning tree of
  $G\setminus x\setminus y$ using all edges incident with a vertex of
  the maximum degree.
  If $G\setminus x\setminus y$ is a cycle, then every vertex is
  deletable and so $G\setminus x\setminus y$ is a path.
  Let $w$ be a degree-$2$ vertex in $G\setminus x\setminus y$.
  Then $G*w$ has at least $3$ deletable vertices and therefore we find
  a desired sequence $v_1,v_2,\ldots,v_n$ such that
  $G*w*v_1*v_2\cdots*v_n$ has an induced path of length $3$ from $x$
  to $y$.
\end{proof}

\begin{LEM}\label{lem:blockmany}
  Let $G$ be a graph and let $x$, $y$ be two vertices in $G$, 
  and let $F_1$, $F_2, \ldots, F_c$ be the components of $G\setminus x\setminus y$.
  If $\rho^*_G(\{x,y\}, F_i)=2$ for all $1\le i\le c$,
  then $G$ has a vertex-minor isomorphic to $\K_{c}\mat\K_{c}$.
\end{LEM}
\begin{proof}
	We proceed by induction on $\abs{V(G)}+\abs{E(G)}$.

        Suppose that $G[V(F_i)\cup \{x,y\}]$ is not an induced path of
        length $3$ from $x$ to $y$.
        By Lemma~\ref{lem:blockone}, there exists a sequence
        $v_1,v_2,\ldots,v_n\in V(F_i)$ such that 
        $G[V(F_i)\cup \{x,y\}]*v_1*v_2\cdots*v_n$ has an induced path of length $3$ from
        $x$ to $y$. If $\abs{V(F_i)}\ge 3$, then we 
        delete all vertices in $F_i$ not on
        this path and apply the induction hypothesis.
        If $\abs{V(F_i)}=2$, then
        $\abs{E(G[V(F_i)\cup \{x,y\}])} > \abs{E(G[V(F_i)\cup \{x,y\}]*v_1*v_2*\cdots
          *v_n)}$  because two vertices in $F_i$ are connected,
        $G[\{x,y\},V(F_i)]$ has at least two edges,
        and $G[V(F_i)\cup \{x,y\}]$ is not an induced path of length
        $3$ from $x$ to $y$.
        So we apply the induction hypothesis to
        $G*v_1*v_2*\cdots*v_n$
        to obtain a vertex-minor isomorphic to $\K_c\mat\K_c$.
        
        Therefore we may assume that  $G[V(F_i)\cup \{x,y\}]$ is  an induced path of
        length $3$ from $x$ to $y$ for all $i$.
        Thus $G*x*y\setminus x\setminus y$ is indeed isomorphic to $\K_c\mat\K_c$.
\end{proof}

\begin{LEM}\label{lem:biparim}
Let $t$ be a positive integer, and $G$ be a bipartite graph with a
bipartition $(S,T)$ such that every vertex in $T$ has degree at least $1$.
Then either $S$ has a vertex of degree at least $t+1$
or $G$ has an induced matching of size at least $\abs{T}/{t}$.
\end{LEM}

\begin{proof}
  We claim that 
  if every vertex in $S$ has degree at most $t$, then 
  $G$ has an induced matching of size at least $\abs{T}/t$.
  We proceed by induction on $\abs{T}$.
  This is trivial if $\abs{T}=0$.
  If $0<\abs{T}\le t$, then we can simply pick an edge to form an
  induced matching of size $1$.
  So we may assume that $\abs{T}>t$.
  
  We may assume that $T$ has a vertex $w$ of degree $1$, because otherwise
  we can delete a vertex in $S$ and apply the induction hypothesis.
  Let $v$ be the unique neighbor of $w$. 
  By the induction hypothesis, $G\setminus v\setminus
  N_G(v)$ has an induced matching $M'$ of size at least
  $(\abs{T}-t)/t$. Now $M'\cup\{vw\}$ is a desired induced matching.
\end{proof}

\begin{LEM}\label{lem:fiber1}
  Let $H$ be a broom in a graph $G$
  having $n$ fibers $F_1,F_2,\ldots,F_n$
  given with $n$ vertices $v_1, v_2, \ldots, v_n$ in
  $V(G)\setminus V(H)$
  such that
  \begin{enumerate}
  \item $v_i$ blocks $F_j$ if and only if $i=j$,
  \item $v_i$ has a neighbor in $F_j$ if and only if $i\le j$.
  \end{enumerate}
  If $n\ge R(c+1,c+1)$, then 
  $G$ has a vertex-minor isomorphic to $P_c$.
\end{LEM}
\begin{proof}
  	If $j>i$, then $v_i$ has a neighbor in $F_j$, but $v_i$ does not block $F_j$. 
  	Therefore, $v_i$ is adjacent to every vertex in $V(F_j)\cap N_H(x)$ for $j>i$.
  	Since $n\ge R(c+1,c+1)$, there exist $1\le t_1< t_2 \cdots < t_{c+1}\le n$ such that
  	$\{v_{t_1}, v_{t_2}, \ldots, v_{t_{c+1}}\}$ is a clique or a stable set of $G$.
  	For $1\le i\le c+1$, let $w_i$ be a vertex in $V(F_{t_i})\cap N_H(x)$. 
	Clearly, \[G[\{v_{t_1}, v_{t_3}, \ldots, v_{t_{2\lceil c/2\rceil -1}}\},\{w_2,w_4, \ldots, w_{2\lceil c/2\rceil}\}]\] is isomorphic to $\S_{\lceil c/2\rceil}\tri\S_{\lceil c/2\rceil}$.
	
        By Lemma~\ref{lem:lengthonecase2}, $\S_{\lceil c/2\rceil}\tri\S_{\lceil c/2\rceil}$ or  
	$\S_{\lceil c/2\rceil}\tri\K_{\lceil c/2\rceil}$ has a vertex-minor isomorphic to $P_{c}$.   
\end{proof}

\begin{LEM}\label{lem:fiber2}
  Let $H$ be a broom in a graph $G$
  having $n$ fibers $F_1,F_2,\ldots,F_n$.
  Let  $v_1, v_2, \ldots, v_n$ be vertices in 
  $V(G)\setminus V(H)$
  such that
  \begin{enumerate}
  \item $v_i$ blocks $F_j$ if and only if $i=j$,
  \item $v_i$ has a neighbor in $F_j$ for all $i$ and $j$.
  \end{enumerate}
  If $n\ge R(c+2,c+2)$, 
  then $G$ has a vertex-minor isomorphic to $\K_c\mat\K_c$.
\end{LEM}
\begin{proof}
    If $i\ne j$, then $v_j$ does not block $F_i$ and therefore $N_G(v_j)\cap V(F_i)=N_G(x)\cap V(F_i)$.
	Since $n\ge R(c+2,c+2)$, there exist $1\le t_1< t_2 \cdots < t_{c+2}\le n$ such that
  	$\{v_{t_1}, v_{t_2}, \ldots, v_{t_{c+2}}\}$ is a clique or a stable set of $G$.

  	We claim that for each $1\le i\le c+2$, 
  	there exist a sequence $w^{(i)}_1,w^{(i)}_2, \ldots,
        w^{(i)}_{k_i}$ of $k_i\ge 0$ vertices in
        $V(F_{t_i})\setminus (N_G(x)\cup N_G(v_{t_i}))$ and $z_i\in V(F_{t_i})$
  	such that $z_i$ is not adjacent to $v_{t_i}$ in $G*w^{(i)}_1*w^{(i)}_2* \cdots* w^{(i)}_{k_i}$
  	but $z_i$ is adjacent to $v_{t_j}$ in $G*w^{(i)}_1*w^{(i)}_2* \cdots* w^{(i)}_{k_i}$ for all $j\neq i$.
  	
  	Let $A^{(i)}_1=(N_G(v_{t_i})\setminus N_G(x))\cap V(F_{t_i})$, 
  	$A^{(i)}_2=(N_G(v_{t_i})\cap N_G(x))\cap V(F_{t_i})$ and $A^{(i)}_3=(N_G(x)\setminus N_G(v_{t_i}))\cap V(F_{t_i})$.

	If $A^{(i)}_3\neq \emptyset$, 
	then a vertex $z_i$ in $A^{(i)}_3$ satisfies the claim. 
	So we may assume $A^{(i)}_3$ is empty.
	Then $A^{(i)}_1\neq \emptyset$ and 
	$A^{(i)}_2\neq \emptyset$,
	otherwise $\rho^*_G(\{v_{t_i},v_{t_j}\},V(F_{t_i}))\le1$ for all $j\neq i$
	because $N_G(v_{t_j})\cap V(F_{t_i})=N_G(x)\cap V(F_{t_i})$. 
	We choose $a^{(i)}_1\in A^{(i)}_1$ and $a^{(i)}_2\in A^{(i)}_2$ so that 
	the distance from $a^{(i)}_1$ to $a^{(i)}_2$ in $F_i$ is minimum.
	
	Let $P_i$ be a shortest path from $a^{(i)}_1$ to $a^{(i)}_2$ in $F_{t_i}$.
	Note that each internal vertex of $P_i$ is not contained in $A^{(i)}_1\cup A^{(i)}_2$.
	After applying local complementation at all internal vertices of $P_i$, 
	$a^{(i)}_1$ is adjacent to $a^{(i)}_2$
        and 
        $v_{t_i}$,
        and non-adjacent to 
        $v_{t_j}$ for all $j\ne i$.
	So by applying one more local complementation at $a^{(i)}_1$
        if necessary,  
        we can delete the edges between $a^{(i)}_2$ and $v_{t_j}$ for all
        $j\ne i$.
	And then, $z_i=a^{(i)}_2$ satisfies the claim.

        Now, take $G'=G*w^{(1)}_1*\cdots *w^{(1)}_{k_1}
        *w^{(2)}_1*\cdots *w^{(2)}_{k_2}
        \cdots
        *w^{(c+2)}_1*\cdots *w^{(c+2)}_{k_{c+2}}$. 
        Since each $w^{(i)}_k$ has no neighbors in $\{v_{t_1}, v_{t_2}, \ldots,
        v_{t_{c+2}}\}$ in $G$,
        applying local complementation at $w^{(i)}_{k}$ does not change the adjacency between any two vertices in $\{v_{t_1}, v_{t_2}, \ldots, v_{t_{c+2}}\}$.
        Thus the induced subgraph of $G'$ on 
        $\{z_1, z_2, \ldots, z_{c+2}\}\cup \{v_{t_1}, v_{t_2}, \ldots,
        v_{t_{c+2}}\}$ is
        isomorphic to $\S_{c+2}\antimat\S_{c+2}$ or $\S_{c+2}\antimat\K_{c+2}$,
	and by Lemma~\ref{lem:lengthonecase1}, $G$ has a vertex-minor isomorphic to $\K_{c}\mat\K_{c}$. 
\end{proof}

\begin{LEM}\label{lem:fiber3}
  Let $H$ be a $(h,n,\ell)$-broom in a graph $G$
  having $n$ fibers $F_1,F_2,\ldots,F_n$
  given with $n$ vertices $v_1, v_2, \ldots, v_n$ in
  $V(G)\setminus V(H)$
  such that
  \begin{enumerate}
			\item $v_i$ blocks $F_j$ if and only if $i=j$,
			\item if $i\neq j$, then $v_i$ has no neighbor in $F_j$.
			\end{enumerate}
	If $n\ge R(t+(c-1)(c-3),c)$, then $G$ has a vertex-minor isomorphic to $P_c$, $\K_c\mat\K_c$, or a $(h,t,\ell +1)$-broom.
\end{LEM}
\begin{proof}
  	Since $n\ge R(t+(c-1)(c-3),c)$, there exist $1\le t_1< t_2 \cdots < t_{k}\le n$ such that
	either
	\begin{enumerate}
  	\item $k=c$ and $\{v_{t_1}, v_{t_2}, \ldots, v_{t_{k}}\}$ is a clique in $G$, or 
  	\item $k=t+(c-1)(c-3)$ and $\{v_{t_1}, v_{t_2}, \ldots, v_{t_{k}}\}$ is a stable set in $G$. 
	\end{enumerate}

	First, we assume that $k=c$ and $\{v_{t_1}, v_{t_2}, \ldots, v_{t_{k}}\}$ is a clique.
	For each $t_i$, 
	since $\rho^*_G(\{x,v_{t_i}\},V(F_{t_i}))\ge 2$, 
	by Lemma~\ref{lem:blockone}, 
	there exists some sequence $w_1, w_2, \ldots, w_n\in V(F_{t_i})$ 
	of (not necessarily distinct) vertices such that
    $G[V(F_{t_i})\cup \{x,v_{t_i}\}]*w_1*w_2\cdots *w_n$ has 
    an induced path of length $2$ from $v_{t_i}$ to $x$. 
    By applying local complementation at $x$, we have a vertex-minor isomorphic to $\K_{c}\mat\K_{c}$.
    
    Now, suppose that $k=t+(c-1)(c-3)$ and $\{v_{t_1}, v_{t_2}, \ldots, v_{t_{k}}\}$ is a stable set in $G$.
    Let $P$ be the handle of $H$.
    If $h+2\ge c$, then we have $P_c$ as an induced subgraph. 
    Thus we may assume that $h\le c-3$.
    We assume that a vertex $y\in V(P)\setminus \{x\}$ adjacent to $c$ vertices in $\{v_1, v_2, \ldots, v_{k}\}$.
    Then since $\rho^*_G(\{x,y\}),V(F_i)\cup \{v_{t_i}\})=2$ for each $i$, 
    by Lemma~\ref{lem:blockmany},
    we have a vertex-minor isomorphic to $\K_c\mat\K_c$.
    Thus, every vertex in the handle other than $x$ 
    cannot have more than $c-1$ neighbors in $\{v_{t_1}, v_{t_2}, \ldots, v_{t_{k}}\}$.
    By deleting at most $(c-1)h$ vertices in $\{v_{t_1}, v_{t_2}, \ldots, v_{t_{k}}\}$, 
    we can remove all edges from $V(P)\setminus \{x\}$ to $\{v_{t_1}, v_{t_2}, \ldots, v_{t_{k}}\}$. 
    Since \[k-(c-1)h\ge k-(c-1)(c-3)\ge t,\]
    we have a vertex-minor isomorphic to a~$(h,t,\ell+1)$-broom. 
\end{proof}

\begin{PROP}\label{prop:incfiber}
For positive integers $c$ and $t$, 
there exists $N=g_2(c,t)$ such that 
for all integers $\ell\ge 2$ and $h\ge 1$, 
every prime graph having a~$(h,N,\ell)$-broom 
has a vertex-minor isomorphic to a~$(h,t,\ell+1)$-broom,
$P_c$, or $\K_c\mat\K_c$. 	
\end{PROP}

\begin{proof}
Let $N=g_2(c,t)=(c-1)m$, where $m=R(m_1, m_2, m_2, m_2)$, 
$m_1=R(t+(c-1)(c-3),c)$, and $m_2=R(c+2,c+2)$.
Let $H$ be a~$(h,N,\ell)$-broom of $G$.
	If a vertex $w$ in $V(G)\setminus V(H)$ blocks $c$ fibers of $H$, then  
	for each fiber $F$ of them, $\rho^*_G(\{w,x\},V(F))=2$. 
	So by Lemma~\ref{lem:blockmany}, $G$ has a vertex-minor isomorphic to $\K_{c}\mat\K_{c}$. 
	Thus, a vertex in $V(G)\setminus V(H)$ can block at most $c-1$ fibers of $H$. 
	
	For each fiber $F$ of $H$, 
	there is a vertex $v\in V(G)\setminus V(H)$ that blocks $F$
	because $G$ is prime.
	Thus, by Lemma~\ref{lem:biparim}, 
	there are ${g_2(c,t)}/{(c-1)}=m$ vertices $v_1, v_2, \ldots, v_{m}$ in $V(G)\setminus V(H)$ and
	 fibers $F_1, F_2, \ldots, F_{m}$ of $H$ 
	such that for $1\le i,j\le m$, $v_i$ blocks $F_j$ if and only if $i=j$.
	For $i\neq j$, either $v_i$ has no neighbor in $F_j$ or $v_i$ has a neighbor in $F_j$ but $\rho^*_G(\{v_i,x\},V(F_j))=1$.

        We assume that $V(K_m)=\{1,2,\ldots,m\}$.
	We color the edges of $K_m$ such that an edge $\{i,j\}$ is 
	\begin{itemize}
	\item green if $N_G(v_i)\cap V(F_j)\neq \emptyset$ and $N_G(v_j)\cap V(F_i)\neq \emptyset$,
	\item red if $N_G(v_i)\cap V(F_j)\neq \emptyset$ and $N_G(v_j)\cap V(F_i)= \emptyset$,
	\item yellow if $N_G(v_i)\cap V(F_j)= \emptyset$ and $N_G(v_j)\cap V(F_i)\neq \emptyset$,
	\item blue if $N_G(v_i)\cap V(F_j)=N_G(v_j)\cap V(F_i)= \emptyset$.
	\end{itemize}
	Since 
	$\abs{V(K_m)}=m=R( m_1,m_2, m_2, m_2 )$, by Ramsey's theorem, 
	either $K_m$ has a green clique of size $m_1$,
	or $K_m$ has a monochromatic clique of size $m_2$ which is red,
        yellow, or blue.
	 
	If $K_m$ has a red clique $C$ of size $m_2$, 
	then for $i,j\in C$, 
	$v_i$ has a neighbor in $F_j$ if and only if $i\le j$.
 	Since $m_2\ge R(c+1,c+1)$, 
 	by Lemma~\ref{lem:fiber1}, 
	 $G$ has a vertex-minor isomorphic to $P_{c}$.

	Similarly, if $K_m$ has a yellow clique $C$ of size $m_2$,
	by Lemma~\ref{lem:fiber1}, 
	$G$ has a vertex-minor isomorphic to $P_c$. 
	 
	If $K_m$ has a blue clique $C$ of size $m_2$,
	then for distinct $i,j\in C$,
	$v_i$ has a neighbor in $F_j$.
	Since $m_2=R(c+2,c+2)$, 
	by Lemma~\ref{lem:fiber2}, 
  	$G$ has a vertex-minor isomorphic to $\K_c\mat\K_c$.
  
   If $K_m$ has a green clique $C$ of size $m_1$, then 
   for distinct $i,j\in C$,
   $v_i$ has no neighbor in $F_j$.
	Since $m_1=R(t+(c-1)(c-3),c)$, 
	by Lemma~\ref{lem:fiber3}, 
	$G$ has a vertex-minor isomorphic to $P_c$, $\K_c\mat\K_c$, or a $(h,t,\ell +1)$-broom.
	\end{proof}	
\subsection{Increasing the height of a broom}

\begin{PROP}\label{prop:nexthandle}
  For positive integers $c$, $ t$, 
  there exists $N=g_3(c,t)$ such that 
  for $h\ge 1$, 
  every prime graph having a~$(h,1,N)$-broom
  has a vertex-minor isomorphic to a~$(h+1,t,1)$-broom or $P_c$.
\end{PROP}

\begin{proof}
       Let $N=g_3(c,t)=g_0(c,2t)$ 
       where $g_0$ is given in Theorem~\ref{thm:binary}.
       Suppose that $G$ has a~$(h,1,N)$-broom $H$
       and let $x$ be the center of $H$.
       Let $F$ be the fiber of $H$.

	Since $F$ is connected, by Theorem~\ref{thm:binary}, $F$ has an induced subgraph isomorphic to $P_c$, or 
	$F$ has a vertex-minor isomorphic to $K_{2t+1}$. 
	We may assume that $F$ has an induced subgraph $F'$ isomorphic to $K_{2t+1}$.  
	Let $P=p_1p_2 \ldots p_m$ be a shortest path from $p_1=x$ to $F'$ in $H$.
	Note that $m\ge 2$ and $p_{m-1}$ is adjacent to at least one vertices of $F'$. 
	Let $S=N_H(p_{m-1})\cap V(F')$.

	We claim that there exists a vertex $v\in V(F')$ such that 
	$(G*v)[V(F)\cup \{x\}]$ has an induced path of length at least $m-1$ from $x$,
	and the last vertex of the path has $t$ neighbors in $F'$ which form a stable set in $G$. 
	
	If $\abs{S}\le t$, then choose $p_{m+1}\in V(F')\setminus S$ and 
	we delete $S\setminus p_m$ from $F'$. 
	And by applying local complementation at $p_{m+1}$, we obtain a path from $x$ to $p_{m+1}$ such that $p_{m+1}$ has $t$ neighbors in $F'$ which form a stable set. 
	
	If $\abs{S}\ge t+1$, then 
	by applying local complementation at $p_m$, we obtain a path from $x$ to $p_{m}$ such that $p_{m}$ has $t$ neighbors in $F'$ which form a stable set. 
	Thus, we prove the claim.
	
	Since $m\ge 2$,
	the union of the handle of $H$ and the path in the claim 
	form a path of length at least $h+1$, 
	and the last vertex of the path has $t$ neighbors which form a stable set in $F'$.
	Therefore, $G$ has a vertex-minor isomorphic to a~$(h+1,t,1)$-broom.
	\end{proof}
 
\begin{PROP}\label{prop:longerhandle}
  For positive integers $c$, $t$,
  there exists $N=g_4(c,t)$ such that 
  for all $h\ge 1$, 
  every prime graph having a~$(h,N,1)$-broom
  has a vertex-minor isomorphic to a~$(h+1,t,1)$-broom,
  $P_c$, or $\K_c\mat\K_c$.
\end{PROP}

\begin{proof}
  By Proposition~\ref{prop:nexthandle}, there exists $N_0$ depending
  only on $c$ and $t$ such that 
  every prime graph having a $(h,1,N_0)$-broom has a vertex-minor
  isomorphic to a $(h+1,t,1)$-broom or $P_c$.
  By applying Proposition~\ref{prop:incfiber} $(N_0-2)$ times, 
  we deduce that there exists $N_1$ such that 
  every prime graph having a $(h,N_1,2)$-broom has a vertex-minor
  isomorphic to a $(h,1,N_0)$-broom, $P_c$, or $\K_c\mat\K_c$.
  By Proposition~\ref{prop:lengthone}, there exists $N$ such that
  every prime graph having a $(h,N,1)$-broom has a vertex-minor 
  isomorphic to a $(h,N_1,2)$-broom, $P_c$, or $\K_c\mat\K_c$.
\end{proof}

We are now ready with all necessary lemmas to prove Proposition~\ref{prop:secexclude}. 

\begin{proof}[Proof of Proposition~\ref{prop:secexclude}]
	By Theorem~\ref{thm:bouchet}, 
        every prime graph on at least $5$ vertices 
        has a vertex-minor isomorphic to $C_5$
        and $P_4$ is a vertex-minor of $C_5$.
        Therefore we may assume that $c\ge 5$.

        By applying Proposition~\ref{prop:longerhandle} $(c-3)$ times, 
        we deduce that 
        there exists a big integer $t$ depending only on $c$ such that
        every prime graph $G$ with a $(1,t,1)$-broom
        has a vertex-minor
        isomorphic to  a     $(c-2,1,1)$-broom, $P_c$, or
        $\K_c\mat\K_c$. 
        Since a $(c-2,1,1)$-broom is isomorphic to $P_c$
        and a $(1,t,1)$-broom is isomorphic to $K_{1,t+1}$, 
        we conclude that every prime graph having 
        a vertex-minor isomorphic to $K_{1,t+1}$
        has  a vertex-minor isomorphic to $P_c$ or $\K_c\mat\K_c$. 
        By Theorem~\ref{thm:boundedsize}, there exists $N$ such that 
        every connected graph on at least $N$ vertices has a
        vertex-minor isomorphic to $K_{1,t+1}$.
        This completes the proof.
\end{proof}

\section{Why optimal?} \label{sec:nonequiv}

Our main theorem (Theorem~\ref{thm:mainthm}) states that sufficiently large prime graphs
must have a vertex-minor isomorphic to $C_n$ or $\K_n\mat \K_n$.
But do we really need these two graphs?
To justify why we need both, 
we should show that for some $n$, $C_n$ is not a
vertex-minor of $\K_N\mat\K_N$ for all $N$
and similarly $\K_n\mat\K_n$ is not a vertex-minor of $C_N$ for all
$N$,
because $C_n$ and $\K_n\mat\K_n$ are also prime.

\begin{PROP}\label{prop:optimal}
  Let $n$ be a positive integer.
  \begin{enumerate}
  \item $\K_3\mat\K_3$ is not a vertex-minor of $C_n$.
  \item $C_7$ is not a vertex-minor of $\K_n\mat\K_n$.
  \end{enumerate}
\end{PROP}
Since $C_7$ is a vertex-minor of $C_n$ for all $n\ge7$, the above
proposition implies that $C_n$ is not a vertex-minor of $\K_N\mat\K_N$
when $n\ge 7$. Similarly $\K_n\mat\K_n$ is not a vertex-minor of $C_N$
for all $n\ge 3$.

We can classify all non-trivial prime vertex-minors of a
cycle graph.
\begin{LEM}\label{lem:vertexminor-cycle}
  If a prime graph $H$ on at least $5$ vertices
  is a vertex-minor of $C_n$, 
  then 
  $H$ is locally equivalent to a cycle graph.
\end{LEM}
\begin{proof}
  We proceed by induction on $n$.
  If $n=5$, then it is trivial. Let us assume $n>5$.
  Suppose $\abs{V(H)}<\abs{V(C_n)}$.
  By Lemma~\ref{lem:bouchet1}, $H$ is a vertex-minor 
  of  $C_n\setminus v$, $C_n*v\setminus v$, or $C_n\pivot vw\setminus v$ for
  a neighbor $w$ of $v$.

  If $H$ is vertex-minor of $C_n*v\setminus v$, then we can apply the
  induction hypothesis because $C_n*v\setminus v$ is isomorphic to $C_{n-1}$.
  
  By Lemma~\ref{lem:primeinduced}, 
  $H$ cannot be a vertex-minor of $C_n\setminus v$
  because  $C_n\setminus v$ has no prime induced subgraph on at least
  $5$ vertices.

  Thus we may assume that 
  $H$ is a vertex-minor of $C_n\pivot vw\setminus v$ for a neighbor $w$ of
  $v$.
  Again, by Lemma~\ref{lem:primeinduced},
  $H$ is isomorphic to a vertex-minor of $C_{n-2}$.
\end{proof}

Classifying prime vertex-minors of $\K_n\mat\K_n$ turns out to be more
tedious. Instead of identifying prime vertex-minors of
$\K_n\mat\K_n$, we focus on characterizing prime vertex-minors on $7$
vertices
to prove (2) of Proposition~\ref{prop:optimal}.

Instead of $\K_n\mat\K_n$, we will first consider $H_n$.
 	Let $H_n$ be the graph having two specified vertices called \emph{roots}
        and $n$ internally disjoint paths of length $3$ joining the roots.
        Let $J_n$ be the graph obtained from $H_n$ by adding %
        a common neighbor of two roots.
        Then $H_n$ has $2n+2$ vertices and $J_n$ has $2n+3$ vertices,
        see Figure~\ref{fig:vmsofhn}.
 	It is easy to observe the following.
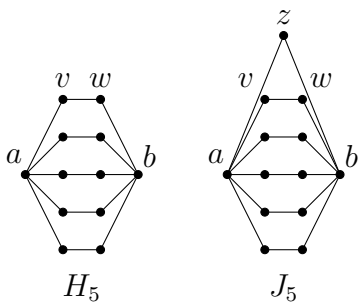
\begin{figure}
 \tikzstyle{v}=[circle, draw, solid, fill=black, inner sep=0pt, minimum width=3pt]
\begin{tikzpicture}[scale=0.05]

 \node[v] at (0,30) {}; 
 \node[v] at (30,30) {}; 
  \foreach \x in {10,20,..., 50}
 {
 \node[v] at (10,\x) {}; 
 \node[v] at (20,\x) {};
 \draw (0,30) -- (10,\x)--(20,\x)--(30,30);
 }
 \node at (-3,35) {$a$}; 
 \node at (10,55) {$v$}; 
 \node at (20,55) {$w$}; 
 \node at (33,35) {$b$};

  \node at (15,0) {$H_5$}; 
\end{tikzpicture}\quad
\begin{tikzpicture}[scale=0.05]

 \node[v] at (0,30) {}; 
 \node[v] at (30,30) {}; 
  \foreach \x in {10,20,..., 50}
 {
 \node[v] at (10,\x) {}; 
 \node[v] at (20,\x) {};
 \draw (0,30) -- (10,\x)--(20,\x)--(30,30);
 }
 \draw (0,30) -- (15,67) -- (30,30);
 \node[v] at (15,67) {};
  \node at (15,0) {$J_5$}; 

 \node at (-3,35) {$a$}; 
 \node at (5,55) {$v$}; 
 \node at (25,55) {$w$}; 
 \node at (33,35) {$b$}; 
 \node at (15,72) {$z$};

\end{tikzpicture}
\caption{The graphs $H_5$ and $J_5$. }\label{fig:vmsofhn}
\end{figure}

\begin{LEM}\label{lem:reducehn}
  Let $H$ be a prime vertex-minor of $H_n$ on at least $5$ vertices.
  If $\abs{V(H_n)}-\abs{V(H)}\ge 3$, 
  then 
  $J_{n-1}$ has a vertex-minor isomorphic to $H$.
\end{LEM}
\begin{proof}
  We may assume $n\ge 3$.
  Since at most $2$ vertices of $H_n$ have degree other than $2$, 
  there exists $v\in V(H_n)\setminus V(H)$ of degree $2$ in $H_n$.
  Let $w$ be the neighbor of $v$ having degree $2$ in $H_n$.
   Let $av'w'b$ be a path of length $3$ from $a$ to $b$ in $H_n$ such that $\{v,w\}\neq\{v',w'\}$. 

	By Lemma~\ref{lem:bouchet1}, 
	$H$ is a vertex-minor of either $H_n\setminus v$, $H_n*v\setminus v$ or $H_n\wedge vw\setminus v$.
        If $H$ is a vertex-minor of $H_n*v\setminus v$, then 
        $H$ is isomorphic to a vertex-minor of $J_{n-1}$, 
	because $H_n*v\setminus v$ is isomorphic to $J_{n-1}$.

	Since $w$ has degree $1$ in $H_n\setminus v$,
	by Lemma~\ref{lem:primeinduced}, 
        if $H$ is a vertex-minor of $H_n\setminus v$, then 
	$H$ is isomorphic to a vertex-minor of $H_n\setminus v\setminus w$.
	Since $H_n\setminus v\setminus w$ is isomorphic to $H_{n-1}$ and $H_{n-1}$ is an induced subgraph of $J_{n-1}$,
	$H$ is isomorphic to a vertex-minor of $J_{n-1}$.
	
	Similarly, if $H$ is a vertex-minor of $H_n\pivot
        vw\setminus v$, then 
	$H$ is isomorphic to a vertex-minor of $H_n\pivot vw\setminus
        v\setminus w$.
        Clearly, $(H_n\pivot vw\setminus v\setminus w)\pivot v'w'$ is isomorphic to $H_{n-1}$. 
        Since $H_{n-1}$ is an induced subgraph of $J_{n-1}$,
        $H$ is isomorphic to a vertex-minor of $J_{n-1}$, as required.
\end{proof}        

\begin{LEM}\label{lem:reducejn}
  Let $H$ be a prime vertex-minor of $J_n$ on at least $5$ vertices.
  If $\abs{V(J_n)}-\abs{V(H)}\ge 2$, 
  then 
  $H_{n}$ has a vertex-minor isomorphic to $H$.
\end{LEM}
\begin{proof}
    We may assume $n\ge 2$.
    Let $a,b$ be the roots of $J_n$, $azb$ be the path of length $2$, 
    and $avwb$ be a path of length $3$ from $a$ to $b$.

    \medskip\noindent
    Case 1: 
    Suppose that $V(J_n)\setminus V(H)$ has a degree-$2$ vertex 
    on a path of length $3$ from $a$ to $b$. 
    We may assume that it is $v$ by symmetry. 
    By Lemma~\ref{lem:bouchet1}, 
    $H$ is a vertex-minor of $J_n\setminus v$, $J_n*v\setminus v$, or 
    $J_n\pivot vw\setminus v$.

    If $H$ is a vertex-minor of $J_n\setminus v$, then 
    $H$ is isomorphic to a vertex-minor of $J_n\setminus v\setminus w$
    by Lemma~\ref{lem:primeinduced}, 
    because $w$ has degree $1$ in $J_n\setminus v$.
    Similarly, if $H$ is a vertex-minor of $J_n\pivot
    vw\setminus v$, then 
    $H$ is isomorphic to a vertex-minor of $J_n\pivot vw\setminus  v\setminus w$.
    Clearly, $J_n\setminus v\setminus w$ and $(J_n\pivot
    vw\setminus v\setminus w)*z$ are isomorphic to $J_{n-1}$,
    and $J_{n-1}$ is a vertex-minor of $H_n$.

    If $H$ is a vertex-minor of $J_n*v\setminus v$, then
    by Lemma~\ref{lem:primeinduced}, 
    $H$ is isomorphic to a vertex-minor of $J_n* v\setminus v\setminus w$, which is isomorphic to $J_{n-1}$,
    because $w$ and $z$ have the same set of neighbors in $J_n* v\setminus v$.
    Since $J_{n-1}$ is a vertex-minor of $H_n$, 
    $H$ is isomorphic to a vertex-minor of $H_n$.
    This proves the lemma in Case 1. 

    \medskip\noindent
    Case 2: Suppose that $z\in V(J_n)\setminus V(H)$.
    Then by Lemma~\ref{lem:bouchet1}, $H$ is a vertex-minor
    of $J_n\setminus z$, $J_n*z\setminus z$, or $J_n\pivot az\setminus
    z$. 
    Since $J_n\setminus z$ and $(J_n*z\setminus z)\pivot vw$ are isomorphic to $H_{n}$,
    we may assume that $H$ is a vertex-minor of $J_n\pivot az\setminus
    z$.
    However, $J_n\pivot az\setminus z$ has no prime induced subgraph
    on at least $5$ vertices and therefore 
    by Lemma~\ref{lem:primeinduced}, 
    $H$ cannot be a vertex-minor of $J_n\pivot az\setminus    z$,
    contradicting our assumption.

    \medskip\noindent
    Case 3: Suppose that $a$ or $b$ is contained in $V(J_n)\setminus V(H)$.
    By symmetry, let us assume $a\in V(J_n)\setminus V(H)$.
    By Lemma~\ref{lem:bouchet1}, $H$ is a vertex-minor
    of $J_n\setminus a$, $J_n*a\setminus a$, or $J_n\pivot az
    \setminus a$.

    Since $J_n\setminus a$ has no prime induced subgraph on at least
    $5$ vertices, $H$ cannot be a vertex-minor of $J_n\setminus a$ 
    by Lemma~\ref{lem:primeinduced}.

	Suppose $H$ is a vertex-minor of $J_n\pivot az\setminus a$.
    By the definition of pivoting, $b$ is adjacent to all vertices of  $N_{J_n}(a)\setminus \{z\}$ in $J_n\pivot az\setminus a$.
    We can remove all these edges between $b$ and $N_{J_n}(a)\setminus \{z\}$ by applying local complementation on all vertices of $N_{J_n}(b)\setminus \{z\}$ in $J_n\pivot az\setminus a$.
    Thus, $H_{n}$ is locally equivalent to $J_n\pivot az\setminus a$,
    and $H$ is isomorphic to a vertex-minor of $H_{n}$.

    Now suppose that $H$ is a vertex-minor of $J_n*a\setminus a$. 
    By the definition of local complementation,
    $N_{J_n}(a)$ forms a clique in $J_n*a\setminus a$.
    So, $b$ is adjacent to all vertices of $N_{J_n}(a)\setminus \{z\}$ in $(J_n*a\setminus a) *z$.
    Similarly in the above case,
    by applying local complementation on all vertices of $N_{J_n}(b)\setminus \{z\}$ in $(J_n*a\setminus a) *z$,
    we can remove all edges between $b$ and $N_{J_n}(a)\setminus \{z\}$
    in $(J_n*a\setminus a) *z$.
    Finally, by pivoting $vw$, we can remove the edge $bz$,
    and therefore, $J_n*a\setminus a$ is locally equivalent to $H_{n}$.
    Thus, $H$ is isomorphic to a vertex-minor of $H_{n}$.     
\end{proof}

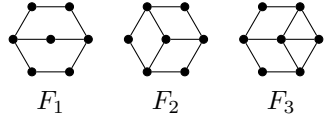
\begin{figure}
 \tikzstyle{v}=[circle, draw, solid, fill=black, inner sep=0pt, minimum width=3pt]
  \newcommand\wheelgraph[2]{\subfloat[#1]{\begin{tikzpicture}[baseline=-3pt]
      \foreach \i in {1,2,...,6} {    \node [v] (v\i) at   (360*\i/6-60:.5)   {} ; }
      \draw (v1)--(v2)--(v3)--(v4)--(v5)--(v6)--(v1);
      \node[v](x) at (0,0) {};
      \foreach \y in {#2}{\draw (x)--(v\y);}
    \end{tikzpicture}}}
 \wheelgraph{$F_1$}{1,4}
 \quad
  \wheelgraph{$F_2$}{1,3,5}
  \quad
  \wheelgraph{$F_3$}{1,2,4,6}
\caption{Graphs $F_1$, $F_2$ and $F_3$. }\label{fig:noteqcycle}
\end{figure}

	Let $F_1,F_2,F_3$ be the graphs in Figure~\ref{fig:noteqcycle}.

\begin{LEM}\label{lem:primevm}
	Let $n\ge 3$ be an integer.
	If a prime graph $H$ is a vertex-minor of $H_n$ and $\abs{V(H)}=7$,
	then $H$ is locally equivalent to $F_1$, $F_2$, or $F_3$.
\end{LEM}

\begin{proof}
  We proceed by induction on $n$. 
  If $n=3$, then 
  let $H$ be a prime $7$-vertex vertex-minor of $H_3$.
  Let $axyb$ be a path  from a root $a$ to the
  other root $b$ in $H_3$.
  By symmetry, we may assume that 
  $V(H_3)\setminus V(H)=\{x\}$ or $\{a\}$.
  By Lemma~\ref{lem:bouchet1}, 
  $H$ is locally equivalent to $H_3\setminus x$, $H_3*x\setminus
  x$, $H_3\wedge xa\setminus x$, 
  $H_3\setminus a$, $H_3*a\setminus a$, or $H_3\wedge ab\setminus a$.
  The conclusion follows because $H_3\setminus x$, $H_3\pivot
  xy\setminus x$, $H_3\setminus a$ are not prime
  and 
  $H_3*x\setminus x$, $H_3\pivot ax\setminus a$, and $H_3*a\setminus
  a$
  are isomorphic to $F_1$, $F_2$, and $F_3$, respectively.

  Suppose $n>3$. 
  By Lemma~\ref{lem:reducehn}, 
  every $7$-vertex prime vertex-minor is
  also isomorphic to a vertex-minor of $J_{n-1}$.
  By Lemma~\ref{lem:reducejn}, 
  it is isomorphic to a vertex-minor of $H_{n-1}$.
  The conclusion follows from the induction hypothesis.
\end{proof}

\begin{LEM}\label{lem:notequiv}
The graphs $F_1$, $F_2$, $F_3$ are not locally equivalent to $C_7$.
\end{LEM}
\begin{figure}
  \centering
  \tikzstyle{v}=[circle, draw, solid, fill=black, inner sep=0pt, minimum width=3pt]
  \tikzstyle{xx}=[ draw, inner sep=3pt, minimum width=3pt]
  \tikzstyle{c}=[circle, draw, inner sep=2pt, minimum width=3pt]
  \newcommand\cyclepicture[1]{  \begin{tikzpicture}[baseline=-3pt]
      \foreach \i in {1,2,...,7} {    \node [v] (v\i) at   (360*\i/7-90:.5)   {} ; }
      \draw (v1)--(v2)--(v3)--(v4)--(v5)--(v6)--(v7)--(v1);
      \node [xx] at (v7) {};
      \foreach \y in {#1} {\node [c] at (v\y) {};}
    \end{tikzpicture}}
  \newcommand\wheelpicture[2]{\begin{tikzpicture}[baseline=-3pt]
      \foreach \i in {1,2,...,6} {    \node [v] (v\i) at   (360*\i/6-60:.5)   {} ; }
      \draw (v1)--(v2)--(v3)--(v4)--(v5)--(v6)--(v1);
      \node[v](x) at (0,0) {};
      \node [xx] at (x) {};
      \foreach \y in {#1}{\draw (x)--(v\y);}
      \foreach \y in {#2} {\node [c] at (v\y) {};}
    \end{tikzpicture}}
  \begin{tabular}{{c|l}}
    $C_7$ \cyclepicture{}&
    \cyclepicture{1,2}\cyclepicture{5,6}\cyclepicture{1,6}\\\hline
    $F_1$ \wheelpicture{1,4}{}&
    \wheelpicture{1,4}{1,4}
  \wheelpicture{1,4}{2,3}
  \wheelpicture{1,4}{5,6}
  \\\hline
  $F_2$ \wheelpicture{1,3,5}{}&
  \wheelpicture{1,3,5}{1,4}
  \wheelpicture{1,3,5}{2,5}
  \wheelpicture{1,3,5}{3,6}
  \\\hline
  $F_3$ \wheelpicture{1,2,4,6}{}&
  \wheelpicture{1,2,4,6}{1,4}
  \end{tabular}

  \caption{List of all $3$-vertex sets having cut-rank $2$ containing a
    fixed vertex $x$ denoted by a square.}

  \label{fig:table}
\end{figure}
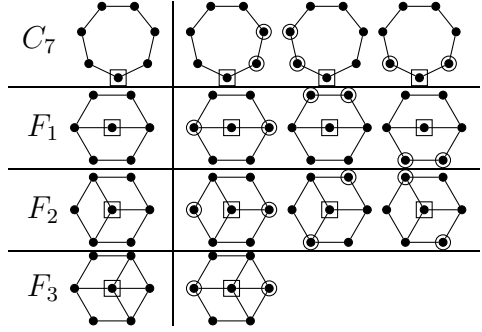

\begin{proof}
        Suppose that $F_i$ is locally equivalent to $C_7$.
        Then $\rho_{F_i}(X)=\rho_{C_7}(X)$ for all $X\subseteq V(C_7)$
        by Lemma~\ref{lem:cutrank}.
        Let $x$ be the vertex in the center of $F_i$, see
        Figure~\ref{fig:table}.
        By symmetry of $C_7$, we may assume that $x$ is mapped to a
        particular vertex in $C_7$. 
	Figure~\ref{fig:table} presents all vertex subsets of size $3$
        having cut-rank $2$ and containing $x$ in graphs $C_7$, $F_1$,
        $F_2$, $F_3$.
        It is now easy to deduce that no bijection on the vertex set
        will map these subsets correctly.
\end{proof}

We are now ready to prove Proposition~\ref{prop:optimal}.
\begin{proof}[Proof of Proposition~\ref{prop:optimal}]
 (1)
 By Lemma~\ref{lem:vertexminor-cycle}, it is enough to check that
 $\K_3\mat\K_3$ is not locally equivalent to $C_6$. This can be
 checked easily.

\smallskip
\noindent       
  (2)
  By applying local complementation at roots, we can easily see that 
  $H_n$ has a vertex-minor isomorphic to $\K_n\mat\K_n$.
        Lemma~\ref{lem:primevm} states that all $7$-vertex prime
        vertex-minors of $H_n$ are $F_1$, $F_2$, and $F_3$. 
        Lemma~\ref{lem:notequiv} proves that none of them are locally
        equivalent to $C_7$.
        Thus $H_n$ has no vertex-minor isomorphic to $C_7$
        and therefore $\K_n\mat\K_n$ has no vertex-minor isomorphic to $C_7$.
\end{proof}

\section{Discussions}
\subsection{Vertex-minor ideals}
A set $I$ of graphs is called  a \emph{vertex-minor ideal}
if for all $G\in I$, all graphs isomorphic to a vertex-minor of $G$
are also contained in $I$.
We can interpret theorems in this paper in terms of vertex-minor
ideals as follows. This formulation allows us to appreciate why these
theorems are optimal.
\begin{COR}
  Let $I$ be a vertex-minor ideal. 
  \begin{description}
    \item  [Theorem~\ref{thm:boundedsize}]
      Graphs in $I$ have bounded number of vertices
      if and only if
      $\{\S_n:n\ge 3\}\not\subseteq I$.
   \item [Theorem~\ref{thm:boundedsize}]
     Connected graphs in $I$ have
     bounded number of vertices 
     if and only if $\{K_n:n\ge 3\}\not\subseteq I$.
   \item [Theorem~\ref{thm:boundedsize}]
     Graphs in $I$ have bounded number of edges
     if and only if 
     $\{K_n:n\ge 3\}\not\subseteq I$
     and $\{\S_n\mat\S_n:n\ge 1\}\not\subseteq I$.
   \item [Theorem~\ref{thm:mainthm}]
     Prime graphs in $I$ have bounded number of vertices
     if and only if 
     $\{C_n:n\ge 3\}\not\subseteq I$
     and $\{\K_n\mat\K_n:n\ge 3\}\not\subseteq I$.
  \end{description}
\end{COR}
\subsection{Rough structure}
We can also regard Theorem~\ref{thm:mainthm} as a rough structure
theorem 
on graphs having no vertex-minor isomorphic to $C_n$ or $\K_n\mat\K_n$
as follows.
The \emph{$1$-join} of two graphs $G_1$, $G_2$ with two specified
vertices $v_1\in V(G_1)$, $v_2\in V(G_2)$
is the graph obtained by making the disjoint union of $G_1\setminus v_1$
and $G_2\setminus v_2$ and adding edges to join neighbors of $v_1$ in
$G_1$ with neighbors of $v_2$ in $G_2$.
\begin{COR}
  For each $n$,  there exists $N$ such that 
  every graph having no vertex-minor isomorphic to $C_n$ or
  $\K_n\mat\K_n$ 
  can be built from graphs on at most $N$ vertices
  by repeatedly taking $1$-join operation. 
\end{COR}

\section*{Acknowledgment}
This research was done while the authors were visiting University of Hamburg.
The authors would like to thank Reinhard Diestel for hosting them.


\begin{thebibliography}{10}

\bibitem{Bouchet1987a}
A.~Bouchet.
\newblock Isotropic systems.
\newblock {\em European J. Combin.}, 8(3):231--244, 1987.

\bibitem{Bouchet1987b}
A.~Bouchet.
\newblock Reducing prime graphs and recognizing circle graphs.
\newblock {\em Combinatorica}, 7(3):243--254, 1987.

\bibitem{Bouchet1988}
A.~Bouchet.
\newblock Graphic presentations of isotropic systems.
\newblock {\em J. Combin. Theory Ser. B}, 45(1):58--76, 1988.

\bibitem{Bouchet1989a}
A.~Bouchet.
\newblock Connectivity of isotropic systems.
\newblock In {\em Combinatorial Mathematics: Proceedings of the Third
  International Conference (New York, 1985)}, volume 555 of {\em Ann. New York
  Acad. Sci.}, pages 81--93, New York, 1989. New York Acad. Sci.

\bibitem{Bouchet1994}
A.~Bouchet.
\newblock Circle graph obstructions.
\newblock {\em J. Combin. Theory Ser. B}, 60(1):107--144, 1994.

\bibitem{CDOV2009}
C.~Chun, G.~Ding, B.~Oporowski, and D.~Vertigan.
\newblock Unavoidable parallel minors of 4-connected graphs.
\newblock {\em J. Graph Theory}, 60(4):313--326, 2009.

\bibitem{Cunningham1982}
W.~H. Cunningham.
\newblock Decomposition of directed graphs.
\newblock {\em SIAM J. Algebraic Discrete Methods}, 3(2):214--228, 1982.

\bibitem{Diestel2010}
R.~Diestel.
\newblock {\em Graph theory}, volume 173 of {\em Graduate Texts in
  Mathematics}.
\newblock Springer, Heidelberg, fourth edition, 2010.

\bibitem{DC2004}
G.~Ding and P.~Chen.
\newblock Unavoidable doubly connected large graphs.
\newblock {\em Discrete Math.}, 280(1-3):1--12, 2004.

\bibitem{DOOV1996}
G.~Ding, B.~Oporowski, J.~Oxley, and D.~Vertigan.
\newblock Unavoidable minors of large {$3$}-connected binary matroids.
\newblock {\em J. Combin. Theory Ser. B}, 66(2):334--360, 1996.

\bibitem{GSH1989}
C.~P. Gabor, K.~J. Supowit, and W.~L. Hsu.
\newblock Recognizing circle graphs in polynomial time.
\newblock {\em J. Assoc. Comput. Mach.}, 36(3):435--473, 1989.

\bibitem{GO2009}
J.~Geelen and S.~Oum.
\newblock Circle graph obstructions under pivoting.
\newblock {\em J. Graph Theory}, 61(1):1--11, 2009.

\bibitem{Geelen1995}
J.~F. Geelen.
\newblock {\em Matchings, matroids and unimodular matrices}.
\newblock PhD thesis, University of Waterloo, 1995.

\bibitem{Naji1985}
W.~Naji.
\newblock Reconnaissance des graphes de cordes.
\newblock {\em Discrete Math.}, 54(3):329--337, 1985.

\bibitem{OOT1993}
B.~Oporowski, J.~Oxley, and R.~Thomas.
\newblock Typical subgraphs of {$3$}- and {$4$}-connected graphs.
\newblock {\em J. Combin. Theory Ser. B}, 57(2):239--257, 1993.

\bibitem{Oum2004}
S.~Oum.
\newblock Rank-width and vertex-minors.
\newblock {\em J. Combin. Theory Ser. B}, 95(1):79--100, 2005.

\bibitem{Oum2006}
S.~Oum.
\newblock Approximating rank-width and clique-width quickly.
\newblock {\em ACM Trans. Algorithms}, 5(1):Art. 10, 20, 2008.

\bibitem{OS2004}
S.~Oum and P.~Seymour.
\newblock Approximating clique-width and branch-width.
\newblock {\em J. Combin. Theory Ser. B}, 96(4):514--528, 2006.

\bibitem{Ramsey1930}
F.~P. Ramsey.
\newblock On a problem of formal logic.
\newblock {\em Proc. London Math. Soc.}, s2-30:264--286, 1930.

\end{thebibliography}
\end{document}